\documentclass[11pt]{amsart}
\usepackage{amsmath,amsthm,amsfonts,latexsym,amscd,amssymb,enumerate,color,hyperref,mathtools,bm}
\usepackage{microtype}
\input{xypic}
\usepackage[margin=2.5cm]{geometry}
\usepackage{rotating}
\usepackage{multirow}
\usepackage{graphicx}
\usepackage{epstopdf}
\usepackage[usenames,dvipsnames]{xcolor}
\usepackage{inputenc}
\usepackage{tikz-cd}
\usepackage{booktabs}
\usepackage{enumitem}
\usepackage{multicol}
\usepackage{etoolbox}
\newcounter{dummy}
\makeatletter  
\newcommand\myitem[1][]{\item[#1]\refstepcounter{dummy}\def\@currentlabel{#1}}   
\makeatother 
\usepackage{hyperref}
\hypersetup{colorlinks,allcolors=blue}
\usepackage{comment}

\usepackage[backend=bibtex, style=alphabetic, sorting=nyt]{biblatex}  
\renewbibmacro{in:}{} 
\def \<{\langle}
\def \>{\rangle}

\newcommand{\RR}{\mathbb{R}}
\newcommand{\NN}{\mathbb{N}}

\newcommand{\PP}{\mathbb{P}}

\newcommand{\calF}{\mathcal{F}}

\newcommand{\W}{\mathrm{W}}

\newcommand{\id}{\mathrm{id}}
\newcommand{\im}{\mathrm{im}}

\newcommand{\eps}{\varepsilon}

\newcommand{\dprod}{d_{H\times Z}}
\DeclareMathOperator{\Int}{Int}

\DeclareMathOperator{\diam}{diam}

\DeclareMathOperator{\Isom}{Isom}

\DeclareMathOperator{\Susp}{Susp}

\DeclareMathOperator{\argmin}{argmin}
\DeclareMathOperator{\supp}{supp}

\DeclareMathOperator{\Conv}{Conv}
\DeclareMathOperator{\Geo}{Geod}
\DeclareMathOperator{\GeodSel}{GeodSel}
\DeclareMathOperator{\Opt}{Opt}

\DeclareMathOperator{\Mid}{Mid}
\DeclareMathOperator{\proj}{proj}
\DeclareMathOperator{\argmax}{argmax}
\DeclareMathOperator{\Cut}{Cut}
\DeclareMathOperator{\RCD}{RCD}

\newcommand{\ppi}{\boldsymbol{\pi}}

\numberwithin{equation}{section}
\numberwithin{table}{section}
\newtheorem{thm}{Theorem}

\newtheorem{cond}{Condition}

\newtheorem{theorem}{Theorem}[section]
\newtheorem{lemma}[theorem]{Lemma}
\newtheorem{proposition}[theorem]{Proposition}
\newtheorem{claim}[theorem]{Claim}
\newtheorem{corollary}[theorem]{Corollary}
\newtheorem{cor}[thm]{Corollary}
\theoremstyle{definition}
\newtheorem{definition}[theorem]{Definition}

\newtheorem{remark}[theorem]{Remark}

\newtheorem*{ack}{Acknowledgements}
\makeatletter
\def\@setthanks{\vspace{-\baselineskip}\def\thanks##1{\@par##1}\thankses}
\makeatother

\title{Isometric Rigidity of Metric Constructions with respect to Wasserstein Spaces}

\author[Che]{Mauricio Che}
\address[Che]{Department of Mathematical Sciences, Durham University, United Kingdom}
\email{mauricio.a.che-moguel@durham.ac.uk}

\author[Galaz-Garc\'ia]{Fernando Galaz-Garc\'ia}
\address[Galaz-Garc\'ia]{Department of Mathematical Sciences, Durham University, United Kingdom}
\email{fernando.galaz-garcia@durham.ac.uk}

\author[Kerin]{Martin Kerin}
\address[Kerin]{Department of Mathematical Sciences, Durham University, United Kingdom}
\email{martin.p.kerin@durham.ac.uk}

\author[Santos-Rodr\'iguez]{Jaime Santos-Rodr\'iguez}
\address[Santos-Rodr\'iguez]{Departamento de Matem\'atica Aplicada, Universidad Polit\'ecnica de Madrid, Spain}
\email{jaime.santos@upm.es}

\begin{document}

\begin{abstract}
    In this paper we study the isometric rigidity of certain classes of metric spaces with respect to the $p$-Wasserstein space. We prove that spaces that split a separable Hilbert space are not isometrically rigid with respect to $\PP_2$. We then prove that infinite rays are isometrically rigid with respect to $\PP_p$ for any $p\geq 1$, whereas taking infinite half-cylinders (i.e.\ product spaces of the form $X\times [0,\infty)$) over compact non-branching geodesic spaces preserves isometric rigidity with respect to $\PP_p$, for $p>1$. Finally, we prove that spherical suspensions over compact spaces with diameter less than $\pi/2$ are isometrically rigid with respect to $\PP_p$, for $p>1$.
\end{abstract}

\maketitle

\setcounter{tocdepth}{1}
\tableofcontents

\section{Introduction}
Given a metric space $X$ and a real parameter $p\geq 1$, the space $\PP_p(X)$ of Borel probability measures on $X$ with finite $p$-moment is equipped with the \emph{$p$-Wasserstein metric} $\W_p$. A natural question arises: \emph{How do the isometries of $\PP_p(X)$ relate to those of $X$?} This question has been addressed in different contexts (see, for example, \cite{BertrandKloeckner2012,BertrandKloeckner2016,GTV20, GeherTitkosVirosztek2022,GTV2023,kloeckner10,Santos2022}). While each isometry of $X$ canonically induces an isometry of $\PP_p(X)$, it is unclear whether $\PP_p(X)$ admits additional, \emph{exotic} isometries. In this paper, we provide both positive and negative answers to this question. We present examples of spaces where exotic isometries exist, as well as examples where $X$ and $\PP_p(X)$ have the same isometry group, a property we term \emph{isometric rigidity of $X$ with respect to $\PP_p$}.

First, we show the existence of exotic isometries in $\PP_2(X)$ for a broad class of non-compact metric measure spaces.

\begin{thm}
    \label{t:flexibility of spaces with lines}
    Let $(Y, d_Y)$ be a proper metric space and $(H,|\cdot|)$ be a separable Hilbert space with $\dim H\geq 1$. Let the product $X = H \times Y$ be equipped with the product metric 
    \[
    d((h_1,y_1), (h_2, y_2)) = \left(|h_1 - h_2|^2 + d_Y(y_1, y_2)^2\right)^\frac{1}{2}.
    \]
    Then $(\PP_2(X), \W_2)$ admits an isometric action of the orthogonal group of $H$ by exotic isometries.
\end{thm}

Theorem~\ref{t:flexibility of spaces with lines} generalises earlier results about Euclidean and Hilbert spaces \cite{kloeckner10,GeherTitkosVirosztek2022}, and provides an infinite family of metric spaces that are not isometrically rigid with respect to $\PP_2$.

A prototypical example satisfying the hypotheses of theorem~\ref{t:flexibility of spaces with lines} is a finite-dimensional, non-negatively curved $\RCD$ space containing a line. The class of non-negatively curved $\RCD$ spaces contains finite-dimensional Alexandrov spaces with non-negative curvature and complete finite-dimensional Riemannian manifolds with non-negative Ricci curvature, with the volume measure as the reference measure. It is well known that such spaces split as a metric product with a Euclidean factor whenever they contain lines (see \cite{CheegerGromoll1971} for Riemannian manifolds,  \cite{milka67} or \cite[Section 10.5]{BBI} for Alexandrov spaces, and \cite{gigli2013,Gigli2018} for $\RCD$ spaces) and that this splitting induces a corresponding splitting of the $2$-Wasserstein space (see \cite{TakatsuYokota2012}, \cite{mitsuishi10} for the Alexandrov setting). 

\begin{cor}\label{c:flexibility of Alexandrov spaces with lines}
    Let $(X,d)$ be a finite-dimensional, non-negatively curved $\RCD$ space that contains a line.  Then $(\PP_2(X), \W_2)$ has an exotic isometry.
\end{cor}

Our second main result complements results in \cite{kloeckner10} and \cite{GTV20}, which fully characterise the isometry groups of Wasserstein spaces over the real line. It also shows that the presence of a ray alone is  not enough  to guarantee the existence of exotic isometries.

\begin{thm}\label{t:rigidity of rays}
The ray $[0,\infty)$ is isometrically rigid with respect to $\PP_p$ for any $p\geq 1$.   
\end{thm}

We then prove that taking half-cylinders over compact, non-branching, geodesic spaces that satisfy a technical assumption (see Condition \ref{cond:manifold point}) preserves isometric rigidity with respect to $\PP_p$ for any $p>1$.

\begin{thm}\label{t:rigidity of cylinders}
Let $(X,d)$ be a compact non-branching geodesic metric space satisfying Condition \ref{cond:manifold point}, and $p,q>1$. Then, if $X$ is isometrically rigid with respect to $\PP_p$ then $X\times_q [0,\infty)=(X\times [0,\infty),d_q)$ is isometrically rigid with respect to $\PP_p$, where the metric $d_q$ is given by
\[
d_q((x_1,t_1),(x_2,t_2)) = (d(x_1,x_2)^q+|t_1-t_2|^q)^{1/q}.
\]
\end{thm}

Riemannian manifolds satisfy Condition \ref{cond:manifold point} (see Remark \ref{rem:manifolds and condition a}). Moreover, as shown in \cite{Santos2022}, closed, positively curved Riemannian manifolds are isometrically rigid with respect to $\PP_2$, while compact rank one symmetric spaces (CROSSes) are isometrically rigid with respect to $\PP_p$ for any $p>1$. This leads to the following result.

\begin{cor}\label{c:cylinders over manifolds}
Let $X$ be closed, positively curved Riemannian manifold. Then $X\times_q [0,\infty)$ is isometrically rigid with respect to $\PP_2$ for any $q>1$. Further, if $X$ is a CROSS, then $X\times_q [0,\infty)$ is isometrically rigid with respect to $\PP_p$ for any $p>1$.
\end{cor}

Finally, we show that taking the spherical suspension over compact, geodesic, non-branching spaces that satisfy a technical assumption (see Condition \ref{cond:general position}) preserves isometric rigidity with respect to $\PP_p$, for any $p>1$.

\begin{thm}\label{t:rigidity of compact spaces}
Let $(X,d)$ be a compact metric space with $\diam(X)<\pi/2$ that satisfies Condition \ref{cond:general position}, and let $p>1$. Then the spherical suspension of $X$ is isometrically rigid with respect to $\PP_p$.
\end{thm}

Since finite-dimensional complete Riemannian manifolds and Alexandrov spaces satisfy  Condition \ref{cond:general position} (see Remark \ref{rem:manifolds and condition b}),  we obtain the following consequence of theorem~\ref{t:rigidity of compact spaces}.

\begin{cor}
If $X$ is a finite-dimensional complete Riemannian manifold or Alexandrov space with $\diam(X)<\pi/2$, then the spherical suspension of $X$ is isometrically rigid with respect to $\PP_p$ for any $p>1$.
\end{cor}

\begin{ack}
This work originated from discussions held during a reading seminar in the Department of Mathematical Sciences at Durham University. The authors would like to thank Mohammad Alattar, Manuel Cuerno, and Kohei Suzuki for their valuable comments. M. C. was funded by CONAHCYT Doctoral Scholarship No. 769708. J. S.-R. was supported in part  by research grants MTM 2017-‐85934-‐C3-‐2-‐P, PID 2021-124195NB-C32 from the Ministerio de Econom\'ia y Competitividad de Espa\~{na} (MINECO) and by a Margarita Salas Fellowship from the Universidad Aut\'onoma de Madrid CA1/RSUE/2021-00625 which allowed him to spend a year as a Visiting Researcher at Durham University.
\end{ack}

\section{Preliminaries}
\label{s:preliminaries}

\subsection{Wasserstein spaces and optimal transport}
We now briefly review some background from the theory of optimal transport. We refer the reader to \cite{ambrosio-gigli,V09} for further details.

Let $(X,d)$ be a metric space and fix $p\in [1,\infty)$. The \emph{$p$-Wasserstein space} $\PP_p(X)$ is the set of Borel probability measures on $X$ with finite $p$-moments, i.e. probability measures $\mu$ on $X$ such that
\[
\int_X d(x,x_0)^p\ d\mu(x) < \infty
\]
for some (and thus for any) $x_0\in X$. The set $\PP_p(X)$ is endowed with the \emph{$p$-Wasserstein metric} $\W_p$ given by
\begin{equation}\label{eq:p-wasserstein}
\W_p(\mu,\nu) = \inf_{\pi} \left(\int_{X\times X} d(x,y)^p\ d\pi(x,y)\right)^{1/p},
\end{equation}
where $\pi$ runs over all possible \emph{transport plans} between $\mu$ and $\nu$, i.e. probability measures $\pi$ on $X\times X$ such that $p^1_\#\pi=\mu$ and $p^2_\#\pi=\nu$. Here, $p^1$ and $p^2$ denote, respectively, the projection maps onto the first and second factors of $X\times X$.

Whenever the infimum in \eqref{eq:p-wasserstein} is achieved, we say that $\pi$ is an {\em optimal plan} between $\mu$ and $\nu$. We denote by $\Opt(\mu,\nu)$ the set of optimal plans between $\mu$ and $\nu$. Observe that this set depends on the parameter $p$, but we omit this dependence for ease of notation and because it will always be clear from the context. Finally, we will some times denote by $\PP_p(Y)$ the set of measures $\mu\in \PP_p(X)$ such that $\supp(\mu)\subset Y$, whenever $Y\subset X$.

Due to general results in the theory of optimal transport, we know that, whenever $X$ is a Polish metric space, i.e. complete and separable, the set $\Opt(\mu,\nu)$ is non-empty for any $\mu,\nu\in\PP_p(X)$ (see, for example, \cite{ambrosio-gigli}).  Moreover, it is possible to characterise optimal plans in terms of \textit{cyclical monotonicity}. Namely, for any transport plan $\pi$ between $\mu,\nu\in\PP_p(X)$, $\pi\in\Opt(\mu,\nu)$ if and only if $\supp(\pi)$ is cyclically monotone, i.e.\ for any $N\in\NN$, $\{(x_1,y_1),\dots,(x_N,y_N)\}\subset \supp(\pi)$, and $\sigma$ a permutation of the set $\{1,\dots,N\}$, the following inequality holds:
\[
\sum_{i=1}^{N} d(x_i,y_i)^p\leq \sum_{i=1}^{N} d(x_i,y_{\sigma(i)})^p.
\]

It is particularly desirable to find optimal plans $\pi \in \Opt(\mu, \nu)$ which are induced by Borel maps $F \colon X \to X$, i.e.~optimal plans $\pi  \in \Opt(\mu, \nu)$ such that $\nu = F_\# \mu$ and $\pi = (\id_X, F)_\# \mu$.  In this case, the map $F \colon X \to X$ is called an \emph{optimal map}.  

In general, $\PP_p(X)$ reflects the geometry of the underlying space $X$. It is well known, for example, that $\PP_p(X)$ is a Polish space whenever $X$ is so, and that $\PP_p(X)$ is compact or geodesic if and only if $X$ has the same property (see, for example, \cite{ambrosio-gigli}). Furthermore, there is an isometric embedding $X\to \PP_p(X)$ given by $x\mapsto \delta_x$, where $\delta_x$ is the Dirac measure at $x$. We denote the image of this embedding by $\Delta_1(X)$. More generally, we define
\[
\Delta_n(X) = \left\{\sum_{i=1}^{n}\lambda_i\delta_{x_i}:n\in \NN,\ x_i\in X,\ \lambda_i\in[0,1],\ \sum_{i=1}^{n} \lambda_i =1 \right\}.
\]
An \emph{atomic measure} is a measure contained in some $\Delta_n(X)$ for some $n\in \NN$. Also, we will omit $X$  and simply write $\Delta_n$ whenever there is no risk of confusion.

An \emph{isometry} of a metric space $Y$ is a surjective, distance-preserving map $\varphi\colon Y\to Y$. We denote the group of isometries of $Y$ by $\Isom(Y)$. The \emph{push-forward map} $\#\colon\Isom(X)\to \Isom(\PP_p(X))$, given by
\[
\varphi \mapsto \varphi_\#,
\]
is an injective homomorphism. Here, $\varphi_\#$ denotes the push-forward by $\varphi$, i.e. $\varphi_\#(\mu)(A)=\mu(\varphi^{-1}(A))$ for any $\mu\in \PP_p(X)$ and any Borel set $A\subset X$.

\begin{definition}\label{d:rigidity and flexibility}
A metric space $X$ is \emph{isometrically rigid with respect to $\PP_p$} if the push-forward map induces an isomorphism of groups, i.e. $\Isom(\PP_p(X))=\#(\Isom(X))$. An isometry $\varphi\in \Isom(\PP_p(X))\setminus\#(\Isom(\PP_p(X))$ is an \emph{exotic isometry}.
\end{definition}

\begin{remark}
Concerning definition \ref{d:rigidity and flexibility}, it is common in the literature to define the \emph{isometric rigidity of $\PP_p(X)$}, rather than the \emph{isometric rigidity of $X$ with respect to $\PP_p$}, as we do. Although the former phrasing simplifies the terminology, the authors prefer the latter as it offers a broader perspective on the rigidity of geometric constructions. This perspective applies not only to the construction of the Wasserstein space but also to other metric spaces constructed functorially from a given metric space.
\end{remark}

To prove theorems~\ref{t:rigidity of rays}, \ref{t:rigidity of cylinders}, and \ref{t:rigidity of compact spaces}, we use the following general strategy. We first show that isometries of the corresponding Wasserstein space preserve the set of Dirac measures. Then, we prove, via different techniques depending on the situation, that the set of atomic measures is also invariant. We then conclude by a density argument and the following characterisation of isometric rigidity.
\begin{proposition}
\label{lem:characterisation of isometric rigidity}
Let $Y$ be a metric space and assume that $\Psi(\Delta_1(Y)) = \Delta_1(Y)$ for any $\Psi\in \Isom(\PP_p(Y))$. Then   $Y$ is isometrically rigid if and only if the subgroup
\[
K = \{\Phi\in \Isom(\PP_p(Y)):\Phi(\delta_{y})=\delta_{y}\ \text{for all}\  y\in Y\} \subset \Isom(\PP_p(Y))\}
\]
is trivial, i.e.\ if the action of $\Isom(\PP_p(Y))$ on $\Delta_1(Y)$ is effective.
\end{proposition}

\begin{proof}
Assume first that $K$ is trivial and let $\Phi\in \Isom(\PP_p(Y))$. Since $\Delta_1(Y)$ is an isometric embedding of $Y$ into $\PP_p(Y)$ and  $\Phi(\Delta_1(Y)) = \Delta_1(Y)$,  the map $\Phi$ induces an isometry $\phi$ of $Y$.  
    Then 
    \[
    \Phi|_{\Delta_1(Y)} = \phi_\#|_{\Delta_1(Y)}.
    \]
    Thus, $\phi_\#^{-1}\circ \Phi \in K$. Since $K$ is trivial, we have $\phi_\#^{-1}\circ \Phi = \mathrm{id}_{\PP_p(Y)}$, which implies $\Phi = \phi_\#$. Therefore, there are no exotic isometries, and $Y$ is isometrically rigid.

    Assume now that $Y$ is isometrically rigid and let $\Phi\in K$. Since $Y$ is rigid, $\Phi=\phi_\#$ for some $\phi\in \Isom(Y)$ and, since $\Phi \in K$, we have $\Phi(\delta_y) =\delta_y$ for any $y\in Y$. Thus,
    \[
    \delta_{\phi(y)} = \phi_\#(\delta_y) = \Phi(\delta_y) = \delta_y,
    \]
    which implies $\phi(y) = y$. Therefore, $\phi= \mathrm{id}_Y$, and it follows that $\Phi =\phi_\# = \mathrm{id}_{\PP_p(Y)}$.
\end{proof}

\subsection{Geodesics, midpoints and non-branching condition}\label{ss:geodesics and curvature bounds} We now briefly recall basic notions on geodesics on metric spaces. The contents of this subsection are taken from \cite{BBI,BridsonHaefliger1999}, wherein further background material can be found. 

Let $(X,d)$ be a metric space. A path $\gamma\colon [0,1]\to X$ is a \emph{geodesic} joining $x, y\in X$ if $\gamma_0=x$, $\gamma_1=y$, and $d(x,y) =  L(\gamma)$, where $L$ denotes the length of $\gamma$. The notation $[xy]$ will sometimes be used for a geodesic joining $x,y$ whenever it is necessary to emphasise the endpoints. Furthermore, we denote by $\Geo(X)$ the set of geodesics (parametrised with constant speed) in $X$, and define the evaluation maps $e_t\colon \Geo(X)\to X$ given by $e_t(\gamma) = \gamma_t$, for each $t\in [0,1]$.

The metric space  $X$ is a \emph{geodesic space} if, for any $x,y\in X$, there exists a geodesic joining $x$ and $y$. Alternatively, we can characterise geodesic spaces in terms of midpoints, whenever $X$ is complete. Namely, $X$ is a geodesic space if and only, for every $x,y\in X$, the set of \emph{midpoints} between $x$ and $y$,
\[
\Mid(x,y) = \{z\in X: d(x,z)=d(y,z)=d(x,y)/2\},
\]
is non-empty. For this characterisation, the assumption of completeness is crucial. We will use this equivalence throughout the article without further comment.

In general, given $x,y\in X$, $x\neq y$, and $t\in [0,1]$, we define
\[
M^t(x,y) = \{ z\in X: d(x,z)=td(x,y)\ \text{and}\ d(y,z)=(1-t)d(x,y)\}.
\]
In particular, $M^{1/2}(x,y)=\Mid(x,y)$ is the set of midpoints between $x$ and $y$. Moreover, we say that $z$ is an interior point of $x,y$ if there exists $t\in (0,1)$ such that $z\in M^{t}(x,y)$. The set of all such  interior points will be denoted by $\Int(x,y)$.

Recall, furthermore, that every geodesic can be reparametrised to have constant speed, i.e. in such a way that $d(\gamma_s,\gamma_t)=|s-t|d(\gamma_0,\gamma_1)$ for any $s,t\in [0,1]$. We denote the set of geodesics on $X$ parametrised with constant speed by $\Geo(X)$, and endow it with the uniform metric.

The space $X$ is \emph{non-branching} if, given any two geodesics $\alpha,\beta\in\Geo(X)$, then having $\alpha|_{[0,t]} = \beta|_{[0,t]}$ for some $t\in (0,1]$ implies $\alpha = \beta$.

For any $x\in X$, the \textit{cut locus} of $x$, denoted by $\Cut(x)$, is the complement of the set of points $y\in X$ such that there exists $\gamma\in \Geo(X)$ with $\gamma_0=x$ and $\gamma_{t}=y$ for some $t\in (0,1)$.

\section{Proof of theorem~\ref{t:flexibility of spaces with lines}: Non-rigidity of spaces splitting a Hilbert space}
\label{s:isometric flexibility}
In this section, we prove theorem~\ref{t:flexibility of spaces with lines}, providing a family of examples that are not isometrically rigid.   We begin with some terminology and preliminary observations.

\begin{definition}
\label{D:split}
Let $(X,d)$ be a Polish metric space and $(H,|\cdot|)$ a Hilbert space.
\begin{enumerate}[label=(\alph*)]
    \item Two isometric embeddings $\eta_1, \eta_2 \colon H \to X$ are \emph{parallel} if there exists an isometric embedding $\iota \colon \im(\eta_1) \cup \im(\eta_2) \to H\times \RR$ of their images such that the image of $\iota$ is of the form $H\times\{s,t\}$ for some $s,t\in \RR$.
    \vspace{1mm}
    \item \label{D:split_c}
    An isometric embedding $\eta \colon H \to X$ in $X$ \emph{induces a splitting} of the Wasserstein space $(\PP_2(X), \W_2)$ if there exists a metric space $(Z,d_Z)$ and an isometry 
    \[
    \Phi \colon (\PP_2(X), \W_2) \to (H\times Z, \dprod)\,,
    \]
    such that the isometric embedding $\overline\eta \colon H \to \PP_2(X)$ in $\PP_2(X)$ defined by $\overline\eta(h) = \delta_{\eta(h)}$ satisfies $\Phi(\overline\eta(h)) = (h, z_o)$ for some $z_o \in Z$, where the product metric $\dprod$ on $H \times Z$ is given by
     \[
     \dprod((h_1, z_1), (h_2, z_2)) = \left( |h_1 - h_2|^2 + d_Z(z_1, z_2)^2 \right)^{\frac{1}{2}}\,.
     \]
    \item The space $X$ has the \emph{$\PP_2$-splitting property} if there exists an isometric embedding $H\colon H\to X$ which induces a splitting of $(\PP_2(X), \W_2)$.
\end{enumerate} 
\end{definition}

\begin{remark}\label{rem:takatsu-yokota}
As mentioned in the introduction, Takatsu and Yokota showed in \cite{TakatsuYokota2012} that if $X$ is a Polish space that splits a separable Hilbert space $H$, then $\PP_2(X)$ splits $H$ as well. Moreover, a consequence of the proof of \cite[Theorem 1.3]{TakatsuYokota2012} is that, in that case, $X$ has the $\PP_2$-splitting property. Observe that \cite[Theorem 1]{mitsuishi10}, combined with \cite[Proposition 2.10]{S06I}, implies this result for non-negatively curved Alexandrov spaces which split a separable Hilbert space.
\end{remark}

Observe that, even without specifying an isometric embedding of $H$, the existence of an isometry $(A,d_A) \to (H \times Z, \dprod)$ (as in part \ref{D:split_c} of Definition \ref{D:split}) ensures that the metric space $(A, d_A)$ is foliated by parallel isometric embeddings of $H$, as shown by the simple lemma below.

\begin{lemma}
\label{L:split}
    Let $(A,d_A)$ and $(Z,d_Z)$ be metric spaces and suppose that there exists an isometry $\Phi \colon (A,d_A) \to (H \times Z, \dprod)$.
    Then:
    \begin{enumerate}
        \item \label{L:lines}
        For all $z \in Z$, the map 
        \begin{align*}
            \eta^z \colon H &\to A \\
            h &\mapsto \Phi^{-1}(h,z)
        \end{align*}
        is an isometric embedding, and for $z_1, z_2 \in Z$, the isometric embeddings $\eta^{z_1}$ and $\eta^{z_2}$ are parallel.
        \vspace{1mm}
        \item  \label{L:lines_b}
        For all $q \in A$, there exists, up to an affine isometry of $H$, a unique isometric embedding $\eta\colon H\to A$ passing through $q$, which is parallel to $\eta^z$ for every $z \in Z$.
        \vspace{1mm}
        \item \label{L:lines_c}
        For all $q \in A$ and for all $z \in Z$, there exists a unique point on $\im(\eta^z)$ closest to $q$, which yields a well-defined projection map $\proj_{\eta^{z}} \colon A \to A$ defined by
        \[
        \proj_{\eta^{z}} \left( \Phi^{-1}(h,w) \right) = \Phi^{-1}(h, z) = \eta^z(h)  \,.
        \]
        \item For all $z,w \in Z$, we have
        \[
        \proj_{\eta^{z}} \circ \proj_{\eta^{w}} = \proj_{\eta^{z}} \,.
        \]
    \end{enumerate}
\end{lemma}

\begin{proof}
The proof follows easily from the existence of the isometry $\Phi$.
\begin{enumerate}
    \item Let $z \in Z$.  Then the map $\ell^z$ is clearly an isometric embedding since
    \[
    d_A(\eta^z(h_1), \eta^z(h_2)) = |h_1 - h_2|
    \]
    by definition. Now, for $z_1, z_2 \in Z$, consider the map
    \[
        \iota \colon \im(\eta^{z_1}) \cup \im(\eta^{z_2}) \to H\times \RR
    \]
    defined by $\iota(\eta^{z_1}(h)) = (h,0)$ and $\iota(\eta^{z_2}(h)) = (h, d_Z(z_1, z_2))$.  
    Moreover, $\iota$ is an isometric embedding, since
    \[
    d_{H\times\mathbb{R}}(\iota(\eta^{z_1}(h_1)),\iota(\eta^{z_2}(h_2)) ) 
    = \left( |h_1 - h_2|^2 + d_Z(z_1, z_2)^2 \right)^{\frac{1}{2}}
    = d_A(\eta^{z_1}(h_1), \eta^{z_2}(h_2)) \,,
    \]
    for all $h_1, h_2 \in H$.  Therefore, $\eta^{z_1}$ and $\eta^{z_2}$ are parallel, as claimed.
    
    \vspace{1mm}
    \item Let $q \in A$ and let $(h_o, z_o) \in H \times Z$, such that $\Phi(q) = (h_o, z_o)$.  By part \ref{L:lines}, $\eta^{z_o}$ is an isometric embedding such that $q\in \im(\eta^{z_o})$ and which is parallel to $\eta^z$ for all $z \in Z$.  
    
    Suppose that $\eta' \colon H \to A$ is another isometric embedding with $q\in \im(\eta')$ which is parallel to $\eta^z$ for all $z \in Z$. Assume also that $\eta'(g_o) = q = \eta^{z_o}(h_o)$.  Clearly, there exist maps $\tau \colon H \to H$ and $\zeta \colon H \to Z$ such that $\Phi(\eta'(g)) = (\tau(g), \zeta(g)) \in H \times Z$, for all $g \in H$.  In particular, $\tau(g_o) = h_o$ and $\zeta(g_o) = z_o$.  

    Since $\eta'$ and $\eta^{z_o}$ are parallel, there exists an isometric embedding $\iota \colon \im(\eta') \cup \im(\eta^{z_o}) \to H\times \RR$ with image of the form $H\times \{\alpha,\beta\}$.  Since $\eta'(g_o) = \eta^{z_o}(h_o)$, it follows that $\alpha = \beta$.  Thus, since $\iota$ is an embedding, $\im(\eta')$ and $\im(\eta^{z_o})$ must coincide and, consequently, $\zeta(g) = z_o$ for all $g \in H$.  If $p^H\colon H\times Z\to H$ is the projection onto the first factor, then, since $p^H|_{H\times \{z_o\}}$ is surjective and $\tau(g)=p^H\circ \Phi \circ \eta'(g)$, we conclude that $\tau$ is an isometry of $H$. Thus, $\tau\colon H\to H$ is a surjective distance-preserving map. By the Mazur--Ulam theorem (see \cite{mazur-ulam}, \cite{Vaisala2003}), $\tau$ is affine. Therefore, $\eta'$ is an affine reparametrisation of $\eta^{z_o}$, as desired.
    
    \vspace{1mm}
    \item Let $q \in A$ and let $(h_o, z_o) \in H \times Z$, such that $\Phi(q) = (h_o, z_o)$.  Then, for every $z \in Z$, we have
    \[
    d_A(q, \eta^z(h)) = \left( |h_o - h|^2 + d_Z(z_o, z)^2 \right)^{\frac{1}{2}} \geq d_Z(z_o, z)\,,
    \]
    with equality if and only if $h = h_o$.  Therefore, there is a unique point on $\im(\eta^z)$ closest to $q$, namely, the point $\eta^z(h_o)$.
    
    \vspace{1mm}
    \item This follows easily from part \ref{L:lines_c}.
\end{enumerate}
\end{proof}

Recall that the Wasserstein space $\PP_2(X)$ consists of Borel probability measures with finite $2$-moments.  A measure with finite $2$-moments naturally yields a well-defined, real-valued function and a notion of center of mass.

\begin{definition}
    Let let $(X,d)$ be a metric space, and $\mu$ a non-negative Borel measure with finite $2$-moments.  
    \begin{enumerate}
        \item The \emph{Fr\'echet function} of $\mu$ is the function
        \begin{align*}
            F_\mu \colon X &\to \RR \\
            x &\mapsto \int_X d(x,y)^2 \ d\mu(y) \,.
        \end{align*}
        \vspace{1mm}
        \item The (possibly empty) set 
        \[
        \calF_\mu =  \argmin(F_\mu ) = \left\{x \in X : F_\mu (x) = \min_{y \in X} \{F_\mu (y)\} \right\} 
        \]
        of minimisers of $F_\mu $ is called the \emph{Fr\'echet-mean set of $\mu$}.  
        \vspace{1mm}
        \item A point $x \in \calF_\mu$, if such a point exists, is called a \emph{Fr\'echet mean}, \emph{barycenter}, or \emph{center of mass} of $\mu$.
    \end{enumerate}
\end{definition}

\begin{remark}\label{rem:existence of frechet means in proper spaces}
In particular, If $\mu = \delta_x$, for $x \in X$, then it is clear that $F_\mu(y) = d(x,y)^2$ and, hence, that $\calF_\mu = \{x\}$.  More generally, it was shown in \cite[Lemma 3.2]{Ohta12} (see also \cite[Theorem 2.1(a)]{bhattacharya-patrangenaru03}) that the Fr\'echet mean set of a probability measure $\mu \in \PP_2(X)$ is non-empty and compact whenever $X$ is a proper metric space.
\end{remark}

As in Euclidean spaces, barycenters in Hilbert spaces are unique. For completeness, we provide a proof of this fact.

\begin{lemma}
\label{L:uniq-bary}
     Let $\mu$ be a non-trivial, finite, non-negative Borel measure on a Hilbert space $H$ with finite $2$-moments. Then $\mu$ has a unique barycenter.
\end{lemma}

\begin{proof}
    Observe that
    \[
    \frac{1}{\mu(H)}F_\mu(x) = |x|^2 - 2B_\mu(x) + C_\mu
    \]
    where 
    \[
    C_\mu=\frac{1}{\mu(H)}\int_H |y|^2\ d\mu(y)
    \] is a positive constant that depends on $\mu$, and $B_\mu$ is the linear and continuous functional given by
    \[
    B_\mu(x) = \frac{1}{\mu(H)}\int_H \langle x,y\rangle\ d\mu(y).
    \]
    By the Riesz representation theorem, there exists a unique $b_\mu\in H$ such that $B_\mu(x) = \langle x,b_\mu\rangle$. Finally, observe that
    \begin{align*}
        \frac{1}{\mu(H)}(F_\mu(x)-F_\mu(b_\mu)) &= |x|^2-2\langle x,b_\mu\rangle + |b_\mu|^2 \geq 0
    \end{align*}
    with equality if and only if $x = b_\mu$. Thus, $b_\mu$ is the unique minimiser of $F_\mu$ and the claim follows.
\end{proof}

From now on, we assume that there exists a proper metric space $(Y, d_Y)$ such that equipping the product $X = H\times Y$ with the product metric yields a Polish metric space $(X,d)$. In particular, by Remark \ref{rem:takatsu-yokota}, for any isometric embedding
\begin{align*}
\eta_{y_o} \colon H &\to X \\
h &\mapsto (h,y_o)
\end{align*}
there exists an isometry
    \[
    \Phi \colon (\PP_2(X), \W_2) \to (H \times Z, \dprod)
    \]
and an element $z_o \in Z$ such that
\begin{align*}
\eta^{y_o} \colon H &\to \PP_2(X) \\
h &\mapsto \delta_{(h,y_o)} = \delta_{\eta_{y_o}(h)}   
\end{align*}
satisfies $\eta^{y_o}(h) = \Phi^{-1}(h,z_o)$.

\begin{lemma}
\label{L:F-mean-set1}
    Let $(X,d)$ be a Polish metric space and $Y$ a proper metric space such that $X = H \times Y$ and $d$ is the product metric.  Fix $y_o \in Y$ and let $\eta_{y_o}$, $\eta^{y_o}$ and $\Phi$ be as defined above.
    Then, for all $\mu \in \PP_2(X)$, the Fr\'echet-mean set $\calF_\mu$ is a non-empty compact subset of $X$ and satisfies
    \[
    \calF_\mu \subset \{g\} \times Y \quad \text{if and only if} \quad \Phi(\mu) \in \{g\} \times Z\,.
    \]
\end{lemma}

\begin{proof}
Let $\mu\in\PP_2(X)$ with $\Phi(\mu) = (g, z_\mu) \in \{g\} \times Z \subset H\times Z$.  By lemma~\ref{L:split}, there exists a unique isometric embedding $\eta^\mu \colon H\to \PP_2(X)$ passing through $\mu$ which is parallel to $\eta^{y_o}$ and such that 
\begin{equation}
\label{eq:mu-line}
\eta^\mu(h) = \Phi^{-1}(h, z_\mu) = \proj_{\eta^\mu}(\eta^{y_o}(h)) = \proj_{\eta^\mu}(\delta_{(h, y_o)})\,.
\end{equation}
In particular, $\mu = \eta^\mu(g)$ and, by a similar argument, for every $y \in Y$ the isometric embedding $\eta^y = \eta^{\delta_{(0,y)}}$ is parallel to $\eta^{y_o}$ and satisfies $\eta^y(h) = \delta_{(h,y)} = \proj_{\eta^{y}}(\delta_{(h,y_o)})$.  

Therefore, again by lemma~\ref{L:split}, we have $\proj_{\eta^{y}}(\mu) = \delta_{(g, y)} = \eta^y(g)$, for every $y \in Y$.  Thus, by definition of the projection, the inequality
\begin{equation}
\label{eq:projection}
\W_2\left(\mu,\delta_{(g,y)}\right) = \W_2\left( \mu, \eta^y(g) \right) \leq \W_2\left( \mu, \eta^y(h) \right) = \W_2\left(\mu,\delta_{(h,y)}\right)
\end{equation} 
holds for all $h \in H$ and all $y \in Y$, with equality if and only if $ h = g$. 

On the other hand, the Fr\'echet function of $\mu$ is given by
\begin{equation}
\label{eq:F-function}
F_\mu(h,y) = \int_{X} d((h,y),(h',y'))^2\ d\mu(h',y') = \W_2(\mu,\delta_{(h,y)})^2 = \W_2(\mu, \eta^y(h))^2\,.    
\end{equation}
Together with \eqref{eq:projection}, this implies that $F_\mu(h,y) \geq F_\mu(g, y)$ for all $(h,y) \in X$, with equality if and only if $h = g$. Therefore, any Fr\'echet mean of $\mu$ is contained in $\{g\} \times Y$. 

Moreover, by the previous argument, if $\{(h_n,y_n)\}_{n\in\NN}$ is a minimising sequence for $F_\mu$, we can assume that $h_n=g$. Since $Y$ is proper, this sequence subconverges to a minimiser of $F_\mu$, that is, a Fr\'echet mean of $\mu$, which implies that $\calF_\mu$ is non-empty. Compactness follows from the properness of $Y$ and the readily verified fact that $\calF_\mu$ is always closed and bounded. 

Conversely, suppose that $\mu \in \PP_2(X)$ and that $\calF_\mu \subset \{g\} \times Y$.  If $\Phi(\mu) \in \{g'\} \times Z$ with $g' \neq g$, then the same argument as above would show that $F_\mu(g',y) \leq F_\mu(h,y)$ for all $h \in H$ and all $y \in Y$, with equality if and only if $h = g'$.  However, this implies that $F_\mu(g',y) < F_\mu(g,y)$ for all $y \in Y$, contradicting the fact that $\calF_\mu \subset \{g\} \times Y$.  Thus, $\Phi(\mu) \in \{g\} \times Z$, as desired.
\end{proof}

Suppose now that $p^H\colon X = H\times Y\to H$ and $p^Y\colon X = H\times Y\to Y$ are the canonical projections, and, for every $\mu \in \PP_2(X)$, let $\mu_H = p^H_\# \mu$ and $\mu_Y = p^Y_\# \mu$ denote the induced push-forward measures on $H$ and $Y$ respectively. Then the Fr\'echet-mean sets of $\mu$ and its push-forward $\mu_H$ are related in the following way.

\begin{lemma}
\label{L:F-mean-set2}
Let $(X,d)$ be a Polish metric space and $Y$ a proper metric space such that $X = H \times Y$ and $d$ is the product metric. 
Then, for all $\mu \in \PP_2(X)$,
    \[
    \calF_\mu \subset \{g\} \times Y \quad \text{if and only if} \quad \calF_{\mu_H} = \{g\}.
    \]
\end{lemma}

\begin{proof}
By lemma~\ref{L:F-mean-set1}, together with \eqref{eq:mu-line}, \eqref{eq:projection} and \eqref{eq:F-function}, the condition $\calF_\mu \subset \{g\} \times Y$ is equivalent to the inequality
\begin{equation}
\label{eq:F-ineq}
F_\mu(g,y) \leq F_\mu(h,y)
\end{equation}
for all $h \in H$ and all $y \in Y$, with equality if and only if $h=g$.  On the other hand, since $\mu_H = p^H_\# \mu$ and $\mu_Y = p^Y_\# \mu$ are the marginals of $\mu$, we have
\begin{align*}
F_\mu(h,y) 
&=\int_{X} d((h,y), (h',y'))^2\ d\mu(h',y')\\
&=\int_{H} |h'-h|^2\ d\mu_H (h')+\int_{Y} d_Y(y',y)^2\ d\mu_Y(y')\\
&= F_{\mu_H}(h) + F_{\mu_Y}(y)\,.
\end{align*}
Therefore, the inequality \eqref{eq:F-ineq} holds if and only if 
\[
F_{\mu_H}(g) \leq F_{\mu_H}(h)
\]
for all $h \in H$, with equality if and only if $h = g$.  Since, by lemma~\ref{L:uniq-bary}, barycenters are unique for elements of $\PP_2(H)$, it follows that this inequality is equivalent to $\calF_{\mu_H} = \{g\}$.
\end{proof}

Kloeckner showed in \cite{kloeckner10} that $\Isom(\PP_2(\RR^n)) \cong \Isom(\RR^n)\ltimes \mathrm{O}(n)$ for $n\geq 2$, where the left factor is given by trivial isometries, i.e.\ maps of the form $\varphi_\#\colon \PP_2(\RR^n)\to \PP_2(\RR^n)$ for some $\varphi \in \Isom(\RR^n)$, and the right factor is given by isometries that fix all Dirac measures. These isometries still have a precise description as follows: for $\psi \in \mathrm{O}(n)$ and $\mu\in \PP_2(\RR^n)$, define 
\begin{equation}\label{eq:exotic action of o(n)}
\Psi(\mu) = (\tau_\mu \circ \psi \circ \tau_\mu^{-1})_\#(\mu), 
\end{equation}
where $\tau_\mu\colon \RR^n \to \RR^n$ is translation by the barycenter $\beta_\mu \in \RR^n$ of $\mu$, which exists and is unique by lemma~\ref{L:uniq-bary}. More generally, Geh\'er, Titkos and Virotsztek proved in \cite{GeherTitkosVirosztek2022} that the group of linear isometries of a separable Hilbert space $H$ acts on $\PP_2(H)$ via exotic isometries, where the action is defined analogously to \eqref{eq:exotic action of o(n)}. In the one-dimensional case, there exist even more isometries than those just described (see \cite{kloeckner10}), however, they are not needed to construct exotic isometries in our setting. 

In preparation for the proof of theorem~\ref{t:flexibility of spaces with lines}, we define the maps that will generate exotic isometries. For a linear isometry $\psi\colon H \to H$, let $\Psi_\psi \colon \PP_2(X)\to \PP_2(X)$ be given by
\begin{equation}
\label{eq:exotic-isom}
\Psi_\psi(\mu) = (\tau_{\mu_H} \circ \psi \circ \tau_{\mu_H}^{-1}\circ p^H,p^Y)_\#(\mu),
\end{equation}
where $\mu_H = p^H_\# \mu \in \PP_2(H)$. By lemmas \ref{L:F-mean-set1} and \ref{L:F-mean-set2}, the map $\tau_{\mu_H} \colon H \to H$ is translation by the barycenter $\beta_{\mu_H} = g \in H$ whenever $\Phi(\mu) \in \{g\} \times Z$.

\begin{proof}[Proof of theorem~\ref{t:flexibility of spaces with lines}]
We will show that the map $\Psi_\psi \colon \PP_2(X) \to \PP_2(X)$ given by \eqref{eq:exotic-isom} is an exotic isometry of $\PP_2(X)$ whenever $\psi\colon H\to H$ is a non-trivial linear isometry.  First, observe that if $(h,y) \in X$ and $\mu = \delta_{(h,y)}$, then $\mu_H = p^H_\# \delta_{(h,y)} = \delta_h \in \PP_2(H)$ has (by lemma~\ref{L:uniq-bary}) its unique barycenter at $t \in H$ and, moreover, $(\tau_{\mu_H} \circ \psi \circ \tau_{\mu_H}^{-1},\id_Y)(h,y) = (h,y)$.  Therefore, the map $\Psi_\psi$ satisfies $\Psi_\psi(\delta_{(h,y)}) = \delta_{(h,y)}$, for all $(h,y) \in X$.  As a result, the only 
measurable map $f\colon X \to X$ that could satisfy $\Psi_\psi = f_\#$ would be the identity map $\id_X$, since $f_\# \delta_{(h,y)} = \delta_{f(h,y)}$ for all $(h,y) \in X$. In this case it would follow that $\Psi_\psi = \id_{\PP_2(X)}$. 

However, it is easy to show that $\Psi_\psi \neq \id_{\PP_2(X)}$.  Indeed, if $v\in H$ is not fixed by $\psi$ and $\mu = \frac{1}{3}\delta_{(0, y)} + \frac{2}{3}\delta_{(v, y)} \in \PP_2(X)$, then $\mu_H = \frac{1}{3}\delta_{0} + \frac{2}{3}\delta_{v}$ and, by (the proof of) lemma~\ref{L:uniq-bary}, $\mu_H$ has its unique barycenter at $\frac{2}{3}v$.
In particular, this implies that 
\[
\tau_{\mu_H} \circ \psi \circ \tau_{\mu_H}^{-1}(0) = \frac{2}{3}v - \frac{2}{3}\psi(v)\neq 0\,.
\]
Therefore, since $\tau_{\mu_H} \circ \psi \circ \tau_{\mu_H}^{-1}$ is invertible,
\[
\Psi_\psi(\mu)\left(\left\{\frac{2}{3}v - \frac{2}{3}\psi(v)\right\} \times Y\right) = \mu \left(\{0\}\times Y\right) = \frac{1}{3} \neq 0 = \mu\left(\left\{\frac{2}{3}v - \frac{2}{3}\psi(v)\right\} \times Y\right)\,.
\]
In other words, the map $\Psi_\psi \colon \PP_2(X) \to \PP_2(X)$ cannot be the identity map and, hence, is not induced by any 
measurable map $f \colon X \to X$.

It remains to show that $\Psi_\psi$ is an isometry.  To that end, let $\mu,\nu \in \PP_2(X)$, and let $\pi\in\Opt(\mu,\nu)$. Define 
\[
\widetilde{\pi} = ((\tau_{\mu_H} \circ \psi \circ \tau_{\mu_H}^{-1}\circ p^H,p^Y)\circ p^1,(\tau_{\nu_H} \circ \psi \circ \tau_{\nu_H}^{-1}\circ p^H,p^Y)\circ p^2)_\#\pi.
\]
A straightforward computation shows that $\widetilde{\pi}$ is a transport plan between $\Psi_\psi(\mu)$ and $\Psi_\psi(\nu)$.

Moreover,
\begin{align*}
\W_2(\Psi_\psi(\mu),\Psi_\psi(\nu))^2 &\leq \int_{X\times X} 
d((h_1,y_1),(h_2,y_2))^2\ d\widetilde{\pi}((h_1,y_1),(h_2,y_2))\\
&= \int_{X\times X} d((\tau_{\mu_H}\circ\psi\circ \tau_{\mu_H}^{-1}(h_1),y_1),(\tau_{\nu_H}\circ\psi\circ \tau_{\nu_H}^{-1}(h_2),y_2))^2\ d\pi((h_1,y_1),(h_2,y_2))\\
&= 
\int_{H\times H} |\tau_{\mu_H} \circ \psi \circ \tau_{\mu_H}^{-1}(h_1)-\tau_{\nu_H} \circ \psi \circ \tau_{\nu_H}^{-1}(h_2)|^2\ d\pi^H(h_1,h_2)\\
&\qquad\qquad +\int_{Y\times Y} d_Y(y_1,y_2)^2\ d\pi^Y(y_1,y_2)\\
&= \int_{H\times H} |h_1-h_2|^2\ d\pi^H(h_1,h_2)+\int_{Y\times Y} d_Y(y_1,y_2)^2\ d\pi^Y(y_1,y_2)\\
&=\int_{X\times X} d_X((h_1,y_1),(h_2,y_2))^2\ d\pi((h_1,y_1),(h_2,y_2))\\
&=\W_2(\mu,\nu)^2,
\end{align*}
where $\pi^H = (p^H\circ p^1,p^H\circ p^2)_\#\pi$ and $\pi^Y = (p^Y\circ p^1,p^Y\circ p^2)_\#\pi$. Observe that we have used \cite[Corollary 3.14]{GeherTitkosVirosztek2022} in the fourth line of the previous chain of equations.

Now, by the same argument applied to $\Psi_{\psi^{-1}}=\Psi^{-1}_\psi$, we conclude that
\[
\W_2(\Psi_\psi(\mu),\Psi_\psi(\nu))=\W_2(\mu,\nu),
\] 
thus establishing the result.
\end{proof}

\section{Proof of theorem~\ref{t:rigidity of rays}: Isometric rigidity of rays}
\label{s:rigidity of rays}
The approach in \cite{kloeckner10} can be adapted to our setting for $p=2$. The proof starts by showing that Dirac measures are invariant via a metric characterisation. We then show that atomic measures with two atoms are also invariant and use the density of their convex hull to complete the proof. The key difference from \cite{kloeckner10} is in how we characterise Dirac measures and atomic measures with two atoms.

For $p\geq 1$ and $p\neq 2$, we adapt the arguments in \cite{GTV20}. Specifically, for $p>1$ and $p\neq 2$, we directly apply Mankiewicz theorem (see \cite[Theorem 3.16]{GTV20}). In the case $p=1$, the main technical ingredient is Claim \ref{claim:characterisation.sigma}, which metrically characterises a set of atomic measures whose invariance implies isometric rigidity. This is an adaptation of \cite[Claim 3.6]{GTV20}.

\begin{lemma}
If $\Phi\in \Isom(\PP_p([0,\infty)))$, then $\Phi(\delta_0)=\delta_0$.
\end{lemma}

\begin{proof}
For any $\mu \in \PP_p([0,\infty))$, $\mu\neq \delta_0$, the unique geodesic from $\delta_0$ to $\mu$ extends past $\mu$ as an infinite ray. However, such  geodesic cannot be extended past $\delta_0$. This metric characterisation of $\delta_0$ ensures that $\delta_0$ is fixed by $\Isom(\PP_p([0,\infty)))$.
\end{proof}

\begin{lemma}\label{l:sigma}
The set 
\begin{equation}\label{eq:sigma}
\Sigma =\{(1-\lambda)\delta_0+\lambda \delta_x:\lambda\in [0,1],\ x\geq 0\}
\end{equation}
is metrically characterised by the following: $\mu\in\Sigma$ if and only if for any maximal infinite ray $(\mu_t)_{t\in [0,\infty)}$ with $\mu_1=\mu$, we have $\mu_0=\delta_0$. 
\end{lemma}

\begin{proof}
First, consider $\mu\in \Sigma$. If $\mu= \delta_0$, the condition holds trivially, since no geodesic rays contain $\delta_0$ in their interior. Now assume $\mu = (1-\lambda)\delta_0+\lambda\delta_x$ for some $\lambda\in (0,1]$ and $x>0$. In this case, $\mu$ is an atomic measure with at most two atoms. If $(\mu_t)_{t\in [0,\infty)}$ is an infinite geodesic passing through $\mu$, then $\mu_0$ is also an atomic measure with at most two atoms (see, for example, \cite[Proposition 2.16]{ambrosio-gigli}). Following a similar argument to that of \cite{kloeckner10}, we prove that $\mu_0=\delta_0$ as follows.

If $\mu_0 \in \Delta_2$ then $\mu_0 = (1-\lambda)\delta_y+\lambda \delta_z$ for some $y,z\in [0,\infty)$. Since the inverse distribution function of $\mu_t$ is well-defined for all $t\geq 0$ and given by
\[
G_{\mu_t}^{-1} = (1-t)G^{-1}_{\mu_0}+tG^{-1}_{\mu_1} = \chi_{(0,1-\lambda]}(1-t)y+\chi_{(1-\lambda,1)}((1-\lambda)z+\lambda x),
\]
where $G^{-1}_{\mu_0}$ and $G^{-1}_{\mu_1}$ are the corresponding inverse distribution functions of $\mu_0$ and $\mu_1$, it follows that $y=0$ or $\lambda = 1$; otherwise, we would have $G^{-1}_{\mu_t}(m)<0$ for any $m\in (0,1-\lambda)$ and $t>1$.

On the other hand, if $\mu\not\in \Sigma$, there exist $0<m_1<m_2<1$ such that $0<G^{-1}_\mu(m_1)<G^{-1}_\mu(m_2)$, where $G^{-1}_\mu$ is the left-continuous inverse distribution function corresponding to $\mu$. For each $t\in[0,\infty)$, let $\mu_t\in \PP_p([0,\infty))$ be defined by the inverse distribution function
\[
G_{\mu_t}^{-1}= \chi_{(0,m_1]}G^{-1}_{\mu}+\chi_{(m_1,1)}((1-t)G^{-1}_\mu(m_1)+tG^{-1}_{\mu}).
\]
It follows that $(\mu_t)_{t\in[0,\infty)}$ is an infinite geodesic ray with $\mu_1=\mu$, but $\mu_0\neq \delta_0$. 
\end{proof}

\begin{lemma}\label{l:invariance of dirac measures}
If $\Phi\in \Isom(\PP_p([0,\infty)))$, then $\Phi|_{\Delta_1}=\id_{\Delta_1}$.    
\end{lemma}

\begin{proof}
Let 
\begin{align}
\Sigma_1 = \{\mu\in \Sigma:\W_p(\mu,\delta_0) = 1\}= \left\{\mu_x=\left(1-\frac{1}{x^p}\right)\delta_0+\frac{1}{x^p}\delta_x:x\geq 1\right\}.
\end{align}
Since both $\Sigma$ and $\{\delta_0\}$ are invariant, it follows that $\Sigma_1$ is also invariant. Clearly, $\delta_1\in \Sigma_1$. A direct computation shows that
\[
\W_p^p(\delta_1,\mu_x) = \left(1-\frac{1}{x^p}\right)+\left(1-\frac{1}{x}\right)^p \quad \text{and}\quad
\W_p^p(\mu_x,\mu_y) = \left(1-\frac{x^p}{y^p}\right)+\left(1-\frac{x}{y}\right)^p
\]
for any $1\leq x\leq y$. In particular, for any $s\in [0,2)$, there is exactly one $x\geq 1$ such that $\W_p^p(\delta_1,\mu_x) = s$. Moreover, since $\W_p^p(\mu_x,\delta_1)=\W_p^p(\mu_x,\mu_{x^2})$ for any $x>1$, it follows that $\delta_1$ is invariant. Furthermore, since for any $t\in (0,\infty)$ there is a unique geodesic ray joining $\delta_0$ with $\delta_t$, which passes through $\delta_1$, and $\PP_p([0,\infty))$ is non-branching, the claim follows.
\end{proof}

\begin{lemma}\label{l:invariance of delta2}
If $\Phi\in \Isom(\PP_p([0,\infty)))$, then $\Phi|_{\Delta_2}=\id_{\Delta_2}$.
\end{lemma}

\begin{proof}
From the proof of lemma~\ref{l:invariance of dirac measures}, we see that $\Sigma_1$ is fixed pointwise. Indeed, since $\delta_1$ is fixed, and for any $s\in [0,2)$ there is exactly one $\mu\in\Sigma_1$ such that $\W_p(\delta_1,\mu) = s$, and because the set $\Sigma_1$ is invariant, it follows that $\Phi(\mu) = \mu$. This also implies that $\Sigma$ is fixed pointwise: for any $\mu\in \Sigma$, there is exactly one geodesic joining $\delta_0$ and $\mu$ which passes through some measure $\widetilde{\mu}\in\Sigma_1$. Since both $\delta_0$ and $\widetilde{\mu}$ are fixed, we conclude that $\mu$ is also fixed.

Finally, if $\mu = (1-\lambda)\delta_a+\lambda\delta_b\in \Delta_2$ with $0<a<b$, then for any $x>b$ there is some $y\in (a,b)$ such that the unique geodesic joining $\widetilde\mu=(1-\lambda)\delta_0+\lambda\delta_x$ and $\delta_y$ passes through $\mu$. Since both $\widetilde\mu$ and $\delta_y$ are fixed, we conclude that $\mu$ is also fixed.
\end{proof}

With these results in hand, we are ready to prove theorem~\ref{t:rigidity of rays}. We consider two cases: $p>1$ and $p=1$.

\begin{proof}[Proof of theorem~\ref{t:rigidity of rays}, case $p>1$]
For $p=2$, understanding how the isometries act on $\Delta_2$ is sufficient, since $\overline{\Conv(\Delta_2)}=\PP_2([0,\infty))$. Let $\mu(x,\sigma,p)$ be the measure in $\Delta_2$ given by
\[
\mu(x,\sigma,p) = \frac{e^{-p}}{e^{-p}+e^{p}}\delta_{x-\sigma e^{p}}+\frac{e^{p}}{e^{-p}+e^{p}}\delta_{x+\sigma e^{-p}}.
\]
A direct computation shows that 
\[
\W_2^2(\mu(x,\sigma,p),\mu(y,\rho,q))=|x-y|^2+\sigma^2+\rho^2-2\sigma\rho e^{|p-q|}.
\]
Then, in particular, any isometry of $\Delta_2$ that fixes $\Delta_1$ is given by
\[
\mu(x,\sigma,p) \mapsto \mu(x,\sigma,\varphi(p))
\] for some $\varphi\in \Isom([0,\infty))$. However, since $\Isom([0,\infty))=\{\id\}$, the only isometry of $\Delta_2$ that fixes $\Delta_1$ is the identity. This implies  
\[\Isom(\PP_2([0,\infty)))=\{\id\}=\#\Isom([0,\infty)).\]

For $p>1$, $p\neq 2$, we can follow \cite{GTV20}. For completeness, we outline the argument here.

The Banach space $L^p(0,1)$ is strictly convex for $p>1$, and the map $\mu\mapsto G_\mu^{-1}$ is an isometric embedding $\PP_p([0,\infty))\to L^p(0,1)$. The image of this embedding is a convex subset containing the $0$ function and its linear span is dense in $L^p(0,1)$. The density follows from the fact that  the image of $\PP_p([0,\infty))$ contains all functions of the form $t\mapsto t^n$, $n\in\NN$. By applying Mankiewicz's theorem (see \cite[Theorem 3.16]{GTV20}), we conclude that any isometry of $\PP_p([0,\infty))$ extends to a distance-preserving map $L^p(0,1)\to L^p(0,1)$ that fixes all constant functions. Since we have already shown that measures of the form $(1-\lambda)\delta_0+\lambda\delta_x$ are fixed, this implies that the indicator functions in $L^p(0,1)$ are also fixed by the (extensions of) isometries of $\PP_p([0,\infty))$. Therefore, the only isometry of $\PP_p([0,\infty))$ is the identity, i.e.\ $\Isom(\PP_p([0,\infty)))= \{\id\}=\#\Isom([0,\infty))$.
\end{proof}

Now we consider the case $p=1$. We adapt the argument given in \cite{GTV20} for the isometric rigidity of $\RR$ with respect to $\PP_1$, and indicate the necessary modifications below.

\begin{proof}[Proof of theorem~\ref{t:rigidity of rays}, case $p=1$]
We first observe that the measure $\delta_0\in \PP_1([0,\infty))$ is fixed. This follows from the fact that $\delta_0$ is the only measure in $\PP_p([0,\infty))$ that is not an intermediate point between any two different measures in $\PP_1([0,\infty))$. Indeed, for any $\mu,\nu\in \PP_1([0,\infty))$, we have
\begin{align*}
    \W_1(\mu,\nu) &= \int_0^1 |G_\mu^{-1}(m)-G_\nu^{-1}(m)|\ dm \\
    &\leq \int_0^1 G_\mu^{-1}(m)\ dm  + \int_0^1 G_\nu^{-1}(m)\ dm\\
    &= \W_1(\mu,\delta_0) + \W_1(\nu,\delta_0)
\end{align*}
and $\W_1(\mu,\nu) = \W_1(\mu,\delta_0) + \W_1(\nu,\delta_0)$ holds if and only if 
\[\int_0^1 \min(G_\mu^{-1}(m),G_\nu^{-1}(m))\ dm = 0.\]
Since both $G_\mu^{-1}$ and $G_\nu^{-1}$ are non-negative, non-decreasing functions, this condition implies that at least one of them is the zero function, meaning $\mu=\delta_0$ or $\nu=\delta_0$.

On the other hand, any measure $\mu\in\PP_1([0,\infty))\setminus\{\delta_0\}$ is a midpoint between $\delta_0$ and the measure $\widetilde\mu$ given by $G^{-1}_{\widetilde\mu} = 2G^{-1}_\mu$. Thus, we have established a geometric characterisation of $\delta_0$, which proves the claim.

We now prove that $\Isom(\PP_1([0,\infty)))$ fixes the set $\Sigma$ defined in \eqref{eq:sigma}. To do so, we consider sets of $t$-intermediate points between measures (see subsection \ref{ss:geodesics and curvature bounds}). According to \cite[Claim 3.5]{GTV20}, we have $\diam(M^{1/2}(\mu,\nu)) = \frac{1}{2}\W_1(\mu,\nu)$ if and only if $\mu$ and $\nu$ are \emph{adjacent}, i.e.\ there exist $a,b\in \RR$, $a<b$, such that $G_\mu|_{\RR\setminus [a,b)} = G_\nu|_{\RR\setminus [a,b)}$ and both $G_\mu|_{[a,b)}$ and $G_\mu|_{[a,b)}$ are constant. Equivalently, $\mu$ and $\nu$ differ at most by the masses they assign to the same atoms $a,b\in \RR$.

\begin{claim}
\label{claim:characterisation.sigma}
Let $\eta\in \PP_1([0,\infty))$. Then $\eta \in \Sigma\setminus\{\delta_0\}$ if and only if, for all $n\in \NN$, there exist measures $\mu_n\in \PP_1([0,\infty))$ satisfying the following conditions:
\begin{enumerate}
\item $\W_1(\delta_0,\mu_n)\to\infty$ as $n\to \infty$.
\item $\delta_0$ and $\mu_n$ are adjacent.
\item $\eta\in M^{1/n}(\delta_0,\mu_n)$.
\item There exists $\eta'\in M^{1/n}(\delta_0,\mu_n)$ such that $\W_1(\eta,\eta')=\diam(M^{1/n}(\delta_0,\mu_n))$.
\end{enumerate}
\end{claim}

To prove the claim, first observe that if $\eta\in \Sigma\setminus\{\delta_0\}$ and $\eta = (1-\lambda)\delta_0+\lambda\delta_x$, a straightforward computation shows that the sequence given by 
\[\mu_n = (1-\lambda)\delta_0+\lambda\delta_{xn}\] satisfies conditions (1)--(3). Furthermore, an argument similar to that in the proof of \cite[Claim 2.3]{GTV20} shows that $\W_1(\eta,\eta')=\diam(M^{1/n}(\delta_0,\mu))$ for \[\eta' = \left(1-\frac{\lambda}{n}\right)\delta_0+\frac{\lambda}{n}\delta_{xn}.\]
Conversely, assume conditions (1)--(4) hold and let $\{\mu_n\}_{n\in\NN}$ be a sequence satisfying these conditions. Then there is a sequence of positive numbers $\{x_n\}_{n\in \NN}$ such that $G_{\mu_n}|_{\RR\setminus [0,x_n)} = G_{\delta_0}|_{\RR\setminus [0,x_n)}\equiv 1$ and such that $G_{\mu_n}|_{[0,x_n)}$ is constant for each $n\in\NN$. In other words, 
\[
\mu_n = \left(1-\lambda_n\right)\delta_0+\lambda_n\delta_{x_n}
\]
for some $\lambda_n\in[0,1]$ such that $\lambda_nx_n\to \infty$. It is then easy to see that if $\eta\in M^{1/n}(\delta_0,\mu_n)$ then $\supp(\eta)\subset [0,x_n)$. By a similar argument similar to that in the proof of \cite[Claim 2.3]{GTV20}, we see that if $\eta,\eta'\in M^{1/n}(\delta_0,\mu_n)$ then 
\[
\W_1(\eta,\eta')=\diam(M^{1/n}(\delta_0,\mu_n))
\]
if and only if 
\[\{\eta,\eta'\} = \left\{\left(1-\frac{\lambda_n}{n}\right)\delta_0+\frac{\lambda_n}{n}\delta_{x_n},\left(1-\lambda_n\right)\delta_0+\lambda_n\delta_{\frac{x_n}{n}}\right\}.
\]
Since the preceding equation holds for each $n\in \NN$ and $\eta$ is fixed, it follows that $\eta=\left(1-\lambda_n\right)\delta_0+\lambda_n\delta_{\frac{x_n}{n}}$ for infinitely many values of $n\in\NN$. This implies that $\lambda_n=\lambda$ and $x_n=nx$ for some constants $\lambda\in (0,1]$, $x\in (0,\infty)$, and infinitely many values of $n\in\NN$. In other words, $\eta\in \Sigma$, proving the claim.

Claim~\ref{claim:characterisation.sigma} implies that $\Sigma$ is invariant under the action of $\Isom(\PP_1([0,\infty)))$, since isometries preserve the adjacency relation (see \cite[Theorem 3.7]{GTV20}) and send $t$-intermediate points to $t$-intermediate points. Moreover, since $\{\delta_0\}$ and $\Sigma$ are invariant, the sets 
\[\Sigma_t=\{\mu\in \Sigma: \W_1(\delta_0,\mu)=t\}\]
are also invariant. An argument analogous to that in the proof of  lemma~\ref{l:invariance of dirac measures} shows that $\Delta_1$ is pointwise fixed by $\Isom(\PP_1([0,\infty)))$. Finally, by a similar argument to that in the proof of \cite[Claim 2.4]{GTV20}, any isometry that fixes all Dirac measures in $\PP_1([0,\infty))$ must be the identity. Thus, we conclude that $\Isom(\PP_1([0,\infty))) = \{\id\} = \#\Isom([0,\infty))$.
\end{proof}

\section{Proof of theorem~\ref{t:rigidity of cylinders}: Isometric rigidity of half-cylinders}\label{s:rigidity of cylinders}

Throughout this section, we assume that $X$ is a non-branching, compact space, and we fix $p,q\in (1,\infty)$. Let us first introduce the spaces we will be working with.

\begin{definition}
Let $(X,d_X), (Y,d_Y)$ be metric spaces. The \emph{$q$-product} of $X$ and $Y$ is the Cartesian product $X\times Y$, equipped with the $q$-product metric given by
\[
d_q((x_1,y_1),(x_2,y_2))= \left(d_X(x_1,x_2)^q+d_Y(y_1,y_2)^q  \right)^{1/q}. 
\]
We denote this space by $X\times_q Y$.
\end{definition}

We now observe how $q$-products interact with the non-branching condition.
\begin{lemma}
\label{l:midpoints}
    Let $X$ and $Y$ be metric spaces. Then $(x,y)\in X\times_q Y$ is a midpoint between $(x_1,y_1),(x_2,y_2)\in X\times_q Y$ if and only if $x\in\Mid(x_1, x_2)$ and $y\in\Mid(y_1,y_2)$. 
\end{lemma}

\begin{proof}
    The ``if'' direction is straightforward. For the ``only if'' part, we note that 
\begin{equation}\label{e:p-mean}
d_X(x_1,x_2)^q\leq (d_X(x_1,x)+d_X(x,x_2))^q\leq 2^{q-1}(d_X(x_1,x)^q+d_X(x,x_2)^q)
\end{equation}
by the triangle inequality in $X$ and the inequality between the arithmetic mean and the $q$-mean of $d_X(x_1,x)$ and $d_X(x,x_2)$. Similarly, we get an analogous inequality involving $y_1,y_2,y$. Adding these inequalities, we get
\begin{equation}\label{e:p-product is non-branching}
d_X(x_1,x_2)^q+d_Y(y_1,y_2)^q\leq 2^{q-1}(d_X(x_1,x)^q+d_X(x,x_2)^q+d_Y(y_1,y)^q+d_Y(y,y_2)^q).
\end{equation}
For $(x,y)\in\Mid((x_1,y_1), (x_2,y_2))$, we have
\begin{align*}
d_X(x_1,x_2)^q+d_Y(y_1,y_2)^q &= 2^{q}(d_X(x_1,x)^q+d_Y(y_1,y)^q) \\
&= 2^{q}(d_X(x,x_2)^q+d_Y(y,y_2)^q) \\
&= 2^{q-1}(d_X(x_1,x)^q+d_Y(y_1,y)^q+d_X(x,x_2)^q+d_Y(y,y_2)^q).
\end{align*}
This implies that the inequalities in \eqref{e:p-mean} and \eqref{e:p-product is non-branching} are equalities. The equality between the arithmetic mean and the $q$-mean implies that $d_X(x_1,x)=d_X(x,x_2)=\frac{1}{2}d_X(x_1,x_2)$. Hence, $x\in\Mid(x_1,x_2)$. Analogously, $y\in\Mid(y_1,y_2)$.
\end{proof}

\begin{lemma}
Let $X$ and $Y$ be metric spaces. Then $X$ and $Y$ are non-branching if and only if $X\times_q Y$ is non-branching for any $q\in (1,\infty)$.
\end{lemma}

\begin{proof}
Let $X$ and $Y$ be non-branching and let $(x_i,y_i),(x,y)\in X\times_q Y$, $i=1,2,3$, be such that $(x,y)\in\Mid((x_1,y_1),(x_2,y_2))\cap\Mid((x_1,y_1),(x_3,y_3))$. By lemma~\ref{l:midpoints},  $x\in\Mid(x_1,x_2)\cap\Mid(x_1,x_3)$, and $y\in\Mid(y_1,y_2)\cap\Mid(y_1,y_3)$. Since $X$ and $Y$ are non-branching,  it follows that $x_2=x_3$ and $y_2=y_3$, so $(x_2,y_2)=(x_3,y_3)$. Thus, $X\times_q Y$ is non-branching. The converse implication follows from the fact that both $X$ and $Y$ are isometrically embedded into $X\times_q Y$.
\end{proof}

\subsection{Invariance of measures on the base of the half-cylinder}
Let us characterise the measures supported on $X\times \lbrace 0\rbrace$, the base of the half-cylinder. In the following, we will need the auxiliary maps 
$T_t\colon X\times_q [0,\infty) \rightarrow X\times_q [0,\infty)$
defined by 
\begin{equation*}
\label{eq:function.T}
    T_t((x,s))=(x,t)
\end{equation*}
for $t\geq 0$.

\begin{lemma}
\label{lem:characterisation.of.measures.on.the.cylinder.base}
A measure $\mu\in\PP_p(X\times_q [0,\infty))$ is supported on $X\times \lbrace 0\rbrace$ if and only if $\mu$ is not in the interior of any maximal ray in $\PP_p(X\times_q [0,\infty))$.    
\end{lemma}

\begin{proof}
Let $\mu\in \PP_p(X\times_q [0,\infty))\setminus \PP_p(X\times \{0\})$. The projection onto the base of the cylinder, $T_0\colon X\times_q [0,\infty)\to X\times \{0\}$, induces an optimal plan between $\mu$ and $(T_0)_\#(\mu)$. Moreover, if $f_t\colon X\times_q [0,\infty)\to X\times_q [0,\infty)$ is given by $f_t(x,s)=(x,ts)$, then we obtain a geodesic between $(T_0)_\#(\mu)$ and $\mu$ by letting 
$\mu_t=(f_t)_\#(\mu)$. This geodesic extends to a maximal ray passing through $\mu$.

Conversely, if $\mu\in \PP_p(X\times \{0\})$ and $(\mu_t)_{t\in[0,\infty)}$ is a maximal ray in $\PP_p(X\times_q[0,\infty))$ that passes through $\mu$, then, by \cite[Theorem 2.10]{ambrosio-gigli}, there is a probability measure $\widetilde\mu\in \PP(\Geo(X))$ such that $\mu_t = e_{t\#}\widetilde\mu$. In particular, for any $\gamma\in \supp(\widetilde\mu)$, we have $\gamma_t \in \supp(\mu_t)$ for all $t\in[0,\infty)$. This implies that $\gamma$ is a constant-speed ray passing through $X\times \{0\}$ at $t=1$. However, such ray cannot exist if $X$ is compact.
\end{proof}

\begin{remark}
Slightly abusing notation, for the remainder of this section, we will denote by $\PP_p(X\times \{t\})$ the set of measures in $\PP_p(X\times_q [0,\infty))$ that are supported on $X\times \{t\}$ for $t\in [0,\infty)$.     
\end{remark}

Lemma~\ref{lem:characterisation.of.measures.on.the.cylinder.base} provides a description of $\PP_p(X\times \{0\})$, based solely on metric properties of $\PP_p(X\times_q [0,\infty))$. Therefore, this set is invariant under the action of the isometry group, leading to the following result.

\begin{proposition}\label{lem:cap of cylinder is invariant}
 Let $\Phi \in \Isom (\PP_p(X\times_q [0,\infty)))$. Then, for all $\mu \in \PP_p(X\times \lbrace 0 \rbrace)$, we have $\Phi (\mu) \in \PP_p (X\times \lbrace 0\rbrace).$   
\end{proposition}

\subsection{Invariance of Dirac measures} 
We will now prove that the set of Dirac measures is invariant under isometries. This is achieved by proving a metric characterisation of Dirac measures. We need to consider the following sets of measures that project to a given measure on the base of the half-cylinder.

\begin{definition}
For any $\nu\in \PP_p(X\times\{0\})$, define 
\[
I(\nu)=\{\mu\in \PP_p(X\times_q [0,\infty)): (T_0)_\#(\mu)=\nu\}.
\]
We now prove the following metric characterisation of $I(\nu)$.
\end{definition}

\begin{lemma}
\label{claim:I(v)}
Let $\nu\in\PP_p(X\times\{0\})$. Then $\mu\in I(\nu)$ if and only if 
$\argmin(\W_p(\mu,\cdot)|_{\PP_p(X\times\{0\}})=\{\nu\}$.
\end{lemma}

\begin{proof}
For any $\nu'\in \PP_p(X\times \{0\})$ and any $\pi\in\Opt(\mu,\nu')$, we have
\begin{align*}
\W_p^p(\mu,\nu') &> \int (d_X(x,y)^q+|s|^q)^{p/q}\ d\pi((x,s),(y,0))\\
&\geq \int |s|^p\ d\pi((x,s),(y,t))\\
&= \int |s|^p\ d\mu(x,s)\\
&\geq \W_p^p(\mu,(T_0)_\#(\mu)).
\end{align*}
Thus,
\begin{equation}\label{e:characterisation sets I-2}
\W_p(\mu,\nu') \geq   \W_p(\mu, (T_0)_\#(\mu))
\end{equation}
which means that $(T_0)_\#(\mu)\in \argmin(\W_p(\mu,\cdot)|_{\PP_p(X\times\{0\}})$ for any $\mu\in \PP_p(X\times_q[0,\infty))$. In particular, if $\argmin(\W_p(\mu,\cdot)|_{\PP_p(X\times\{0\}})=\{\nu\}$, then $\nu = (T_0)_\#(\mu)$. Hence, $\mu\in I(\nu)$.

Conversely, if $\mu\in I(\nu)$, then $\nu=(T_0)_\#(\mu)\in \argmin(\W_p(\mu,\cdot)|_{\PP_p(X\times\{0\}})$. For any $\nu'\in \PP_p(X\times \{0\})$, we have $\nu'\in\argmin(\W_p(\mu,\cdot)|_{\PP_p(X\times\{0\}})$ if and only if equality holds in \eqref{e:characterisation sets I-2}, which implies  $d_X(x,y)=0$ for $\pi$-a.e. $((x,s),(y,0))$ whenever $\pi\in \Opt(\mu,\nu')$. Therefore,
$\nu' = (T_0)_\#(\mu) = \nu$. 
\end{proof}

\begin{lemma}\label{lem:invariance of sets I}
Let $\nu\in\PP_p(X\times\{0\})$ and $\Phi\in\Isom(\PP_p(X\times_q [0,\infty)))$. Then $\Phi(I(\nu))=I(\Phi(\nu))$.
\end{lemma}

\begin{proof}
Let $\Phi\in \Isom(\PP_p(X\times_q [0,\infty))$. By proposition~\ref{lem:cap of cylinder is invariant}, $\nu \in \PP_p(X\times \{0\})$ if and only if $\Phi(\nu)\in \PP_p(X\times \{0\})$. Since $\Phi$ is an isometry,  
\[
\Phi(\argmin(\W_p(\mu,\cdot)|_{\PP_p(X\times\{0\})})) = \argmin(\W_p(\Phi(\mu),\cdot)|_{\PP_p(X\times\{0\})}).
\] 
By lemma~\ref{claim:I(v)}, $\mu\in I(\nu)$ if and only if $\Phi(\mu)\in I(\Phi(\nu))$, and the conclusion follows.
\end{proof}

 We can use the sets $I(\nu)$ to characterise measures in $\Delta_1(X\times\{0\})$ as follows.

\begin{lemma}\label{lem:characterisation of dirac measures}
Let $\mu\in\PP_p(X\times\{0\})$. Then the set $I(\mu)$ is geodesic (with respect to the restricted metric from $\PP_p(X\times[0,\infty))$) if and only if $\mu \in \Delta_1(X\times\lbrace 0 \rbrace).$
\end{lemma}

\begin{proof}
Fix $x\in X$ and consider the map $i_x\colon [0,\infty) \rightarrow X\times [0,\infty)$ given by $i_x(t) = (x,t)$.  Observe that $i_x$ is an isometric embedding of $[0,\infty)$ into $X\times [0,\infty)$ with image $\{x\}\times [0,\infty)$. Therefore, $i_{x\#}\colon \PP_p([0,\infty))\to \PP_p(X\times [0,\infty))$ is an isometry onto its image.
We now show that $i_{x\#}( \PP_p([0,\infty))= I(\delta_{(x,0)})$. 
Indeed, for any $\mu\in \PP_p([0,\infty))$, we have
\[
T_{0\#}i_{x\#}\mu = (T_0\circ i_x)_\#\mu = c_{(x,0)\#}\mu = \delta_{(x,0)},
\]
where $c_{(x,0)}\colon [0,\infty)\to X\times \{0\}$ is the constant map $c_{(x,0)}(t) = (x,0)$. Hence, $i_{x\#}\mu\in I(\delta_{(x,0)})$.

Conversely, if $\nu\in I(\delta_{(x,0)})$, define $\mu\in\PP_p([0,\infty))$ by
\[
\mu(S) = \nu(i_x(S))
\]
for any Borel measurable set $S\subset [0,\infty)$. Then, for any Borel measurable set $E\subset X\times [0,\infty)$,
\[
i_{x\#}\mu (E) = \mu(i_x^{-1}(E)) = \nu(i_x(i_x^{-1}(E))) = \nu(E\cap \{x\}\times [0,\infty)) = \nu(E),
\]
which means that $i_{x\#}\mu = \nu$. Therefore, $i_{x\#}( \PP_p([0,\infty))= I(\delta_{(x,0)})$, which implies that $I(\delta_{(x,0)})$ is geodesic.

Now, take $\mu \in \PP_p(X\times\lbrace 0\rbrace)\setminus \Delta_1(X\times\{0\})$. We will show that $I(\mu)$ is not geodesic. Specifically, we will find two measures $\nu_0$ and $\nu_1$ in $I(\mu)$ such that any geodesic joining them is not contained in $I(\nu)$.

Since $\mu$ is not a Dirac measure, $\diam(\supp(\mu)) >0$. Given that $X$ is compact, we may select $(x,0),(y,0)\in \supp(\mu)$ such that 
\[
d_X(x,y) = d_p((x,0),(y,0)) = \diam (\supp(\mu)).
\]

Let $B_x$ and $B_y$ be balls around $(x,0)$ and $(y,0)$, respectively, such that any geodesic joining  points in $B_x$ does not intersect $B_y$ and any geodesic joining points in $B_y$ does not intersect $B_x$. Set $C= X\setminus B_x\cup B_y$. Since $B_x$ and $B_y$ are disjoint, we can decompose $\mu$ as 
\[
\mu = \mu\llcorner B_x + \mu\llcorner C +\mu\llcorner B_y.
\]
Without loss of generality, assume $\mu(B_x)\leq \mu(B_y)$. 
Set $D=\diam(\supp(\mu))$ and consider the measures
\[
\nu_0 = T_{2D\#}\mu\llcorner B_x + T_{4D\#}\mu\llcorner C+ \mu\llcorner B_y 
\]
and
\[
\nu_1 = \mu\llcorner B_x + T_{4D\#}\mu\llcorner C+ T_{2D\#}\mu\llcorner B_y,
\]
which clearly belong to $I(\mu)$, since $T_{0\#}(\nu_0)=T_{0\#}(\nu_1) = \mu$ (see figure~\ref{fig:rigidity of cylinders}).

Having defined the measures $\nu_0$ and $\nu_1$, we now show that any geodesic joining them must leave $I(\mu)$.
Let $\pi\in\Opt(\nu_0,\nu_1)$ and suppose there exist points
$((u,4D),(v,\alpha))$ and $((w,\beta),(z,4D))$ in $\supp(\pi)$ with $\alpha,\beta\in \lbrace 0,2D\rbrace$. Then
\begin{align*}
    d_X^p(u,z)+d_X^p(w,v)+|\beta-\alpha|^p &\leq 2(\diam X)^p+(2\diam X)^p\\
    &\leq d_X^p(u,v)+|4D-\alpha|^p+d_X^p(w,z)+|4D-\beta|^p,
\end{align*}
which contradicts the fact that  $\supp(\mu)$ must be cyclically monotone. Thus, $\pi|_{T_{4D}(C)\times T_{4D}(C)}$ is supported on the diagonal of $T_{4D}(C)\times T_{4D}(C)$.

Similarly, $\pi$ must move mass between $T_{2D}(B_x)$ and $T_{2D}(B_y)$.
That is, we cannot only have points in $\supp(\pi)$ of the form $((a,2D),(b,0)), ((d,0),(e,2D))$ with $a,b\in B_x, d,e \in B_y.$ To see this, notice that
\[
d_X^p(a,e)+d_X^p(d,b)\leq 2(\diam X)^p\leq d_X^p(a,b)+(2D)^p+d_X^p(d,e)+(2D)^p,
\]
which again contradicts cyclical monotonicity. 

\begin{figure}
    \centering
    \def\svgwidth{0.4\linewidth}
    %% Creator: Inkscape 1.3.2 (091e20e, 2023-11-25, custom), www.inkscape.org
%% PDF/EPS/PS + LaTeX output extension by Johan Engelen, 2010
%% Accompanies image file 'rigidity_of_cylinders.pdf' (pdf, eps, ps)
%%
%% To include the image in your LaTeX document, write
%%   \input{<filename>.pdf_tex}
%%  instead of
%%   \includegraphics{<filename>.pdf}
%% To scale the image, write
%%   \def\svgwidth{<desired width>}
%%   \input{<filename>.pdf_tex}
%%  instead of
%%   \includegraphics[width=<desired width>]{<filename>.pdf}
%%
%% Images with a different path to the parent latex file can
%% be accessed with the `import' package (which may need to be
%% installed) using
%%   \usepackage{import}
%% in the preamble, and then including the image with
%%   \import{<path to file>}{<filename>.pdf_tex}
%% Alternatively, one can specify
%%   \graphicspath{{<path to file>/}}
%% 
%% For more information, please see info/svg-inkscape on CTAN:
%%   http://tug.ctan.org/tex-archive/info/svg-inkscape
%%
\begingroup%
  \makeatletter%
  \providecommand\color[2][]{%
    \errmessage{(Inkscape) Color is used for the text in Inkscape, but the package 'color.sty' is not loaded}%
    \renewcommand\color[2][]{}%
  }%
  \providecommand\transparent[1]{%
    \errmessage{(Inkscape) Transparency is used (non-zero) for the text in Inkscape, but the package 'transparent.sty' is not loaded}%
    \renewcommand\transparent[1]{}%
  }%
  \providecommand\rotatebox[2]{#2}%
  \newcommand*\fsize{\dimexpr\f@size pt\relax}%
  \newcommand*\lineheight[1]{\fontsize{\fsize}{#1\fsize}\selectfont}%
  \ifx\svgwidth\undefined%
    \setlength{\unitlength}{196.71174057bp}%
    \ifx\svgscale\undefined%
      \relax%
    \else%
      \setlength{\unitlength}{\unitlength * \real{\svgscale}}%
    \fi%
  \else%
    \setlength{\unitlength}{\svgwidth}%
  \fi%
  \global\let\svgwidth\undefined%
  \global\let\svgscale\undefined%
  \makeatother%
  \begin{picture}(1,1.47340485)%
    \lineheight{1}%
    \setlength\tabcolsep{0pt}%
    \put(0,0){\includegraphics[width=\unitlength,page=1]{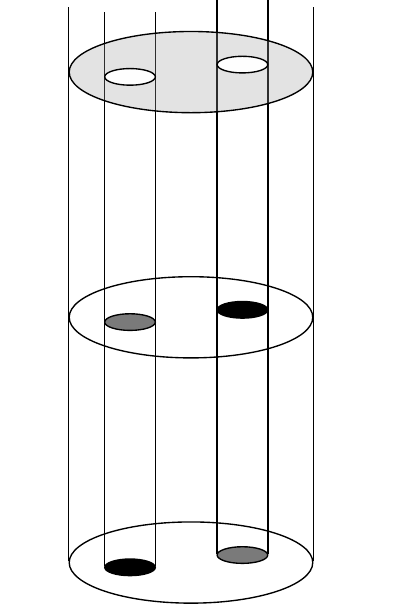}}%
    \put(0.80073257,0.07887659){\color[rgb]{0,0,0}\makebox(0,0)[lt]{\lineheight{1.25}\smash{\begin{tabular}[t]{l}$X\times \{0\}$\end{tabular}}}}%
    \put(0.80075252,0.68126789){\color[rgb]{0,0,0}\makebox(0,0)[lt]{\lineheight{1.25}\smash{\begin{tabular}[t]{l}$X\times \{2D\}$\end{tabular}}}}%
    \put(0.80075232,1.27604671){\color[rgb]{0,0,0}\makebox(0,0)[lt]{\lineheight{1.25}\smash{\begin{tabular}[t]{l}$X\times \{4D\}$\end{tabular}}}}%
    \put(0,0){\includegraphics[width=\unitlength,page=2]{rigidity_of_cylinders.pdf}}%
    \put(0.00013371,0.67745203){\color[rgb]{0,0,0}\makebox(0,0)[lt]{\lineheight{1.25}\smash{\begin{tabular}[t]{l}$\phantom{X\times \{2D\}}$\end{tabular}}}}%
  \end{picture}%
\endgroup%

    \caption{In this diagram, the deep grey regions represent $\supp(\nu_0)\cap X\times\{0,2D\}$; the dark regions represent $\supp(\nu_1)\cap X\times \{0,2D\}$; and pale grey region on top represent $\supp(\nu_0)\cap X\times\{4D\}=\supp(\nu_1)\cap X\times\{4D\}$. The arrows represent any optimal plan between $\nu_0$ and $\nu_1$; any such optimal plan fixes the mass on $X\times \{4D\}$.}
    \label{fig:rigidity of cylinders}
\end{figure}

Therefore, if $(\nu_t)_{t\in [0,1]}$ is the geodesic joining $\nu_0$ and $\nu_1$, by \cite[Theorem 2.10]{ambrosio-gigli} we have $\nu_t=(e_t)_\#\Gamma$ for some measure $\Gamma$ on $\Geo(X)$ such that $(e_0,e_1)_\#\Gamma = \pi$. In particular, for $t\in (0,1)$, and due to the previous arguments, 
\begin{align*}
\nu_t(C\times [0,\infty)) &= \Gamma(e^{-1}_t(C\times [0,\infty)))\\
&= \Gamma(e^{-1}_t(C\times\{4D\}))+\Gamma(e^{-1}_t(C\times [0,2D])) \\
&= \mu(C) + \Gamma(e^{-1}_t(C\times [0,2D]))\\
&> \mu(C) = \nu_0(C\times [0,\infty)).
\end{align*}
However, if $\nu_t\in I(\mu)$, we would have $\nu_t(C\times [0,\infty))= \nu_0(C\times [0,\infty)$. 
\end{proof}

We now show the invariance of the set of Dirac measures under isometries of $\PP_p(X\times_q [0,\infty))$.

\begin{proposition}\label{prop:invariance.of.dirac.measure}
If $\Phi \in \Isom (\PP_p(X\times_q[0,\infty)))$, then $\Phi(\Delta_1(X\times_q[0,\infty)))= \Delta_1(X\times_q[0,\infty))$
\end{proposition}

\begin{proof}
Let $\Phi\in\Isom(\PP_p(X\times_q[0,\infty)))$. It is sufficient to prove that 
\[\Phi(\Delta_1(X\times[0,\infty))\subset\Delta_1(X\times [0,\infty)),\]
since the reverse inclusion follows from considering $\Phi^{-1}$. We consider two cases:
Dirac measures of the form $\delta_{(x,0)}$, for some $x\in X$, and arbitrary Dirac measures $\delta_{(z,t)}$ with $(z,t)\in X\times[0,\infty)$.

First, consider $\delta_{(x,0)}$ for some $x\in X$. By proposition~\ref{lem:cap of cylinder is invariant}, $\Phi(\delta_{(x,0)})\in \PP_p(X\times\{0\})$. Moreover, by lemmas \ref{lem:invariance of sets I} and \ref{lem:characterisation of dirac measures}, $\Phi(\delta_{(x,0)})$ is also a Dirac measure supported on $X\times \{0\}$ and the conclusion follows. 

Now, consider $\delta_{(z,t)}$ with $(z,t)\in X\times [0,\infty)$. Then $\delta_{(z,t)}\in I(\delta_{(z,0)})$, and from the previous case, we know $\Phi(\delta_{(z,0)})= \delta_{(y,0)}$ for some $y\in X$. By lemma~\ref{lem:invariance of sets I}, it follows that $
\Phi(\delta_{(z,t)})\in I(\delta_{(y,0)})$.
Furthermore, recall that the maps $i_y,i_z\colon [0,\infty)\to X\times [0,\infty)$ given by $t\mapsto (y,t)$ and $t\mapsto (z,t)$, respectively, are isometric embeddings. Therefore, $i_{y\#}^{-1}\circ \Phi \circ i_{z\#}\in \Isom (\PP_p([0,\infty))).$ By theorem~\ref{t:rigidity of rays}, we have
$i_{y\#}^{-1}\circ \Phi \circ i_{z\#} = \id$, which implies that $\Phi(\delta_{(z,t)})=\delta_{(y,t)}$, yielding the result.
\end{proof}

\begin{remark}
    Note that, under the hypotheses of theorem~\ref{t:rigidity of cylinders}, we may prove lemma~\ref{prop:invariance.of.dirac.measure} using a simpler argument. Indeed, if $\Phi\in \Isom(\PP_p(X\times [0,\infty))$, by proposition~\ref{lem:cap of cylinder is invariant}, $\Phi|_{\PP_p(X\times \{0\})}$ is an isometry of $\PP_p(X\times \{0\})$.  Assuming that $X$ is isometrically rigid, it follows that, for some $\phi \in \Isom(X)$, $\Phi(\mu) = \phi_\#(\mu)$ whenever $\mu\in\PP_p(X\times \{0\})$. In particular, $\Phi(\delta_{(x,0)})=\delta_{(\phi(x),0)}$ for any $(x,0)\in X\times \{0\}$. The rest of the argument is as in the second paragraph of the proof of lemma~\ref{prop:invariance.of.dirac.measure}.
\end{remark}

\subsection{Invariance of atomic measures} 
So far, we have only assumed that $X$ is compact, geodesic, and non-branching. To establish the invariance of atomic measures and achieve the desired isometric rigidity, we will now assume that $X$ satisfies the following technical condition for the remainder of the section.

\begin{cond}\label{cond:manifold point}
For any $N\in \NN$ and $\{x_1,\dots,x_N\}\subset X$, there exists $x_o\in X$ such that $\{x_1,\dots,x_N\}$ is contained in the set 
\begin{align*}
J(x_o)=\,&\lbrace y\in X: \text{there exists } \gamma\in \Geo(X) \text{ with } \gamma_0= y, \gamma_{t_o}= x_o \text{ for some } t_o\in (0,1)   \rbrace\\
=\,&\{y \in X: x_o\in\Cut(y)\}.
\end{align*}
\end{cond}

\begin{remark}\label{rem:manifolds and condition a}
   Condition \ref{cond:manifold point} holds for Riemannian manifolds. Indeed, for any $\{x_1,\dots,x_N\}\subset X$, we know that $X\setminus (\Cut(x_i)\cup \{x_i\})$ is open and dense for $i= 1,\dots, N$ (see subsection \ref{ss:geodesics and curvature bounds}). Therefore, $\bigcap_{i=1}^{N}X\setminus (\Cut(x_i)\cup \{x_i\})$ is open and dense. In particular, for any $x_o\in\bigcap_{i=1}^{N}X\setminus (\Cut(x_i)\cup \{x_i\})$, the required condition holds.
\end{remark}

We will use the following lemma to establish the invariance of atomic measures.

\begin{lemma}\label{prop.invariant-totallyatomic-cylinder}
Assume that $X$ satisfies Condition \ref{cond:manifold point}, and let $\Phi\in \Isom(\PP_p(X\times_p[0,\infty)))$ be such that $\Phi|_{\Delta_1}=\id_{\Delta_1}$. For $n\geq 2$, the set $\Delta_{n}(X\times_q [0,\infty))$, consisting of atomic measures with at most $n$ atoms, is invariant under $\Phi$.  
\end{lemma}

\begin{proof}
Take $\mu = \sum_{i=1}^n \lambda_i\delta_{(x_i,t_i)}$ such that $t_i\neq t_j$ for all $i\neq j$ and let $x_o\in X$ be such that $x_i\in J(x_o)$ for all $i$, as in Condition \ref{cond:manifold point}. Such measures are dense in $\Delta_n(X\times_q [0,\infty))$. Consider the map $S\colon X\times_q [0,\infty) \rightarrow X \times_q [0,\infty) $ given by $S(x,t)= (x_o,t)$. Note that $S_{\#}\mu\in I(\delta_{(x_o,0)})\cap \Delta_n(X\times_q [0,\infty))$. We now show that $\W_p(\mu,S_{\#}\mu)=\W_p(\mu,I(\delta_{(x_o,0)})).$  Take $\overline{\mu}\in I(\delta_{(x_o,0)})$ and $\pi\in \Opt(\mu,\overline{\mu})$. Then, 
\begin{align*}
\int d_q^p((x,t),(x_o,s))\ d\pi ((x,t),(x_o,s)) &\geq \int d_X^p(x,x_o)\ d\pi((x,t),(x_o,s))\\  
&= \int d_q^p((x,t),(x_o,s))\ d(\id, S)_{\#}\mu((x,t),(x_o,s)),
\end{align*}
which implies 
\[
\W_p(\mu,\overline{\mu}) \geq \W_p(\mu,S_\#\mu).
\]

Since $x_i\in J(x_o)$ for $i=1,\dots, n$, there exist points $z_i\in X$ and geodesics $\gamma^i$ such that $\gamma_0^i = x_i$, $\gamma_1^i = z_i$, and $\gamma^i$ passes through $x_o$. 
Let $(\mu_t)_{t\in [0,1]}$ be the geodesic joining $\mu$ and $S_{\#}\mu$ with mass moving along the geodesics $s\mapsto (\gamma^i_s,t_i)$.
Pick $t_o\in (0,1)$ so that $\max d_X(\gamma_{t_o}^i, x_o) <\frac{1}{2}\min_{i\neq j}|t_i-t_j|$. 
Then the geodesic $(\eta_t)_{t\in [0,1]}$ joining $\mu_{t_o}$ and $\sum_{i=1}^n \lambda_i\delta_{(z_i,t_i)}$ moves along the geodesics $s\mapsto (\gamma^i_s,t_i)$, $1\leq i\leq n$. 
Moreover, $(\eta_t)_{t\in [0,1]}$ passes through $S_{\#}\mu$. 

Since $S_{\#}\mu \in I(\delta_{(x_o,0)})\cap \Delta_n(X\times[0,\infty))$, we have $\Phi(S_{\#}\mu)=S_{\#}\mu$, which implies that $\Phi(\eta_t)\in \Delta_n(X\times_q [0,\infty))$ for all $t\in [0,1].$ In particular, $\Phi(\mu_{t_0})\in \Delta_n(X\times_q [0,\infty))$, so  $\Phi(\mu)\in \Delta_n(X\times_q [0,\infty))$, as desired.
\end{proof}

Combining lemma~\ref{prop.invariant-totallyatomic-cylinder} and proposition~\ref{lem:characterisation of isometric rigidity}, we obtain the invariance of atomic measures.

\begin{corollary}
Assume that $X$ satisfies the hypotheses of theorem~\ref{t:rigidity of cylinders}, and let $n\in\NN$. Then the set of atomic measures $\Delta_n(X\times [0,\infty))$ is invariant under the action of $\Isom(\PP_p(X\times_q [0,\infty))$.
\end{corollary}

Without assuming much about $X$, we have already gathered useful information on the behaviour of isometries of $\PP_p(X\times_q[0,\infty))$. To prove the rigidity of $X\times_q [0,\infty)$, we will now assume that $X$ is isometrically rigid. With the necessary results in place, we now turn to the proof of theorem~\ref{t:rigidity of cylinders}. 

\begin{proof}[Proof of theorem~\ref{t:rigidity of cylinders}]
Let $n\geq 2$ and consider a measure $\mu\in \PP_p(X\times_q [0,\infty))$ given by 
\begin{equation}
\label{eq:atomic.measures.in.general.position}
\mu = \sum_{i=1}^n \lambda_i\delta_{(x_i,t_i)},\quad x_i\neq x_j,\ t_i\neq t_j,\ \lambda_i\neq\lambda_j,\ \lambda_i\in (0,1)\text{ for all } i\neq j.
\end{equation}
It is clear that such measures are dense in $\Delta_n(X\times_q [0,\infty))$. Now, fix $z\in X$. As in the proof of lemma~\ref{prop.invariant-totallyatomic-cylinder}, we have 
\[
\argmin\left(\W_p(\mu,\cdot)|_{I(\delta_{(z,0)})}
\right)
= \left\{\sum_{i=1}^n \lambda_i\delta_{(z,t_i)}\right\}.
\]
Moreover, for any $\Phi \in K$ (with $K$ as in proposition~\ref{lem:characterisation of isometric rigidity}), the isometry $\Phi$ restricted to $I(\delta_{(z,0)})$ is the identity. Therefore,
\[
\argmin\left(\W_p(\Phi(\mu),\cdot)|_{I(\delta_{(z,0)})}\right)= \left\{\sum_{i=1}^n\lambda_i\delta_{(z,t_i)}\right\}.
\]
This implies that $\Phi(\mu) = \sum_{i=1}^n\lambda_i\nu_i$, where $\nu_i$ is supported on $X\times\lbrace t_i\rbrace$ for $i=1,\cdots,n$. Furthermore, by lemma~\ref{prop.invariant-totallyatomic-cylinder}, each $\nu_i$ is a Dirac measure. Hence, $\Phi(\mu)$ is an atomic measure with at most $n$ atoms. In other words,
\[
\Phi(\mu) = \sum_{i=1}^{n}\lambda_i\delta_{(y_i,t_i)}
\]
for some $y_i\in X$, $i=1,\dots,n$. Moreover, 
\[
\Phi(\mu) \in I\left(\Phi\left(\sum_{i=1}^n\lambda_i\delta_{(x_i,0)}\right)\right)
\]
and, since $X$ is isometrically rigid, we have
\[
\Phi\left(\sum_{i=1}^n\lambda_i\delta_{(x_i,0)}\right)= \sum_{i=1}^n\lambda_i\delta_{(x_i,0)}.
\]
Therefore, $\Phi(\mu)$ has exactly $n$ atoms of the form $(x_i,t_{\sigma(i)})$ for some permutation $\sigma$ of $\{1,\dots,n\}$. Additionally, we have
\[
\Phi(\mu)(\lbrace x_i\rbrace\times[0,\infty))=\Phi(\mu)(X\times\{t_i\})=\lambda_i.
\]
Hence, 
\[
\Phi(\mu)= \sum_{i=1}^n\lambda_i\delta_{(x_i,t_i)}=\mu.
\]
By the density of measures of the form \eqref{eq:atomic.measures.in.general.position}, every measure in $\Delta_n(X\times_q [0,\infty))$ is fixed and, since $n$ is arbitrary, we conclude that $\Phi = \id$. By proposition~\ref{lem:characterisation of isometric rigidity}, the half-cylinder $X\times_q [0,\infty)$ is isometrically rigid.
\end{proof}

\section{Proof of theorem~\ref{t:rigidity of compact spaces}: Isometric rigidity of spherical suspensions}\label{s:rigidity of compact spaces}
Let $X$ be a compact metric space. In this section, we prove that $\Susp(X)$, the spherical suspension of $X$, is isometrically rigid with respect to $\PP_p$ for any $p>1$, provided that the diameter of $X$ is less than $\pi/2$. Note that we do not require the existence of geodesics in $X$, as the geometry of the spherical suspension  provides enough structure.

Recall that $\Susp(X)$, the \emph{suspension} of $X$, is the quotient space
\[
\Susp(X) = (X\times [0,\pi])/\sim,
\]
where $(x,0)\sim(x',0)$ and $(x,\pi)\sim (x',\pi)$ for any $x,x'\in X$. To ease the notation, we let $\mathbf{0}= [x,0]$ and ${\ppi} = [x,\pi].$ We will also denote by $X\times\{t\}$ the set $(X\times \{t\})/\sim$, for any $t\in(0,\pi)$. We endow $\Susp(X)$ with the spherical suspension metric
\[
d_{\Susp}([x,t],[y,s]) = \cos^{-1}\left(\cos(t)\cos(s)+\sin(t)\sin(s)\cos(d(x,y)) \right).
\] 
The metric space $(\Susp(X),d_{\Susp})$ is the \emph{spherical suspension} of $X$.

We refer the reader to \cite[Propositions 5.9 and 5.10-2)]{BridsonHaefliger1999} for the proof of the following lemma on basic metric properties of $\Susp(X)$. 

\begin{lemma}\label{prop.propertiessuspension}
The spherical suspension $\Susp(X)$ has the following properties:
\begin{enumerate}
    \item $(\Susp(X), d_{\Susp})$ is a compact metric space.
    \item  Let $[x,t] \in \Susp(X)$ with $t\in(0,\pi)$. Then there exists a unique geodesic $\gamma\colon [0,1]\rightarrow \Susp(X)$ such that $\gamma_0 ={\mathbf{0}},$ $\gamma_1 = {\ppi},$ and $\gamma_{s}= [x,t]$
    for some $s \in [0,1].$
\end{enumerate}
\end{lemma}

Recall that, since $X$ is compact, $\PP_p$ is compact, $p\geq 1$ (see, for example, \cite[Remark~7.1.9]{AGS08}). The next lemma characterises the measures that realise the diameter of $\PP_p(\Susp(X))$.

\begin{lemma}\label{lemma.diametersuspension}
Let $p>1.$ Then
\begin{enumerate}
    \item\label{lemma.diametersuspension.1} $\diam(\Susp(X))= \diam(\PP_p(\Susp(X)))= \pi,$
    \item\label{lemma.diametersuspension.2} $d_{\Susp}([x,t],[y,s])= \pi$ if and only if $\{[x,t],[y,s]\}=\{{\mathbf{0}},{\ppi}\}$.
    \item\label{lemma.diametersuspension.3} If $\mu, \nu \in \PP_p(\Susp(X))$, then $\W_p(\mu,\nu)= \pi$ if and only if $\{\mu,\nu\}=\{\delta_{\mathbf{0}}, \delta_{\ppi}\}$.
\end{enumerate}
\end{lemma}
\begin{proof}
For (1), let $\mu,\nu\in \PP_p(\Susp(X)).$ Then $\mu\otimes\nu$ is a transport plan between $\mu$ and $\nu$. Therefore,
\[
\W_p^p(\mu,\nu) \leq \int d_{\Susp}^p([x,s],[y,t])\ d(\mu\otimes\nu)([x,s],[y,t]) \leq \pi^p.
\]
Taking $\mu=\delta_{\mathbf{0}}$ and $\nu=\delta_{{\ppi}}$ yields the equality.

Now we prove (2). Let $[x,s],[y,t]$ be points in $\Susp(X)$ with 
\[
d_{\Susp}([x,s],[y,t]) = \pi = \diam(\Susp(X)).
\]
Using the definition of the spherical suspension metric, we obtain that
\[
\cos(s)\cos(t)+\sin(s)\sin(t)\cos(d(x,y)) = -1,
\]
which implies that $\{s,t\}=\{0,\pi\}$. The converse implication is clear.

Finally, we prove (3). Take $\mu,\nu \in \PP_p(\Susp(X))$ such that $\W_p(\mu,\nu)=\pi.$ Observe that
\begin{equation}
\label{eq:inequalities.item.(3)}
\pi^{p} =\W_p^p(\mu,\nu)\leq \int d_{\Susp}^p([x,s],[y,t])\ d(\mu\otimes\nu)([x,s],[y,t])  \leq \pi^p.
\end{equation}
This implies that $\mu\otimes \nu$ is optimal and that all inequalities in \eqref{eq:inequalities.item.(3)} are equalities.  
In particular, $d_{\Susp}^p([x,s],[y,t])=\pi$ for any $([x,s],[y,t])\in \supp(\mu\otimes \nu)$. Therefore, by item (2),   $\{\mu,\nu\}=\{\delta_{\mathbf{0}},\delta_{\ppi}\}$. The converse implication is clear.
\end{proof}

We will use the following result in the spirit of \cite[Theorem 2.10]{ambrosio-gigli}. Recall that we denote by $e_t\colon \Geo(X)\to X$ the evaluation map $\gamma\mapsto \gamma_t$, for each $t\in [0,1]$ (see section \ref{s:preliminaries}).

\begin{lemma}
\label{teo.geodSusp}
Let $\mu_0,\mu_1\in \PP_p(\Susp(X))$ and assume that there exists $\pi\in\Opt(\mu_0,\mu_1)$ such that, for every $([x,t],[y,s])\in \supp(\pi)$, there exists a geodesic in $\Susp(X)$ joining $[x,t]$ and $[y,s]$. Then there exists a geodesic in $\PP_p(\Susp(X))$ joining $\mu_0$ with $\mu_1.$ Furthermore, the following assertions are equivalent:
\begin{enumerate}
    \item $t\mapsto \mu_t$ is a geodesic in $\PP_p(\Susp(X))$.
    \item There exists a measure $\eta\in\PP_p(\Geo(\Susp(X)))$ such that $(e_0,e_1)_{\#}\eta\in\Opt(\mu_0,\mu_1)$ and \[\mu_t=(e_t)_\#\eta.\]
\end{enumerate}
\end{lemma}

\begin{proof}
Let 
\[
E = \lbrace (([x,t],[y,s]),\gamma) : ([x,t],[y,s])\in\Susp(X)^2,\ \gamma\in \Geo(\Susp(X)),\ \gamma_0 = [x,t] \text{ and }  \gamma_1=[y,s] \rbrace.
\]
Observe that $E$ is a closed subset of $\Susp(X)^2\times \Geo(\Susp(X))$, where both $\Susp(X)^2$ and $\Geo(\Susp(X))$ are Polish spaces. Then using classical measurable selection principles (see, for example, \cite[Theorem 1]{A74}), we can find a Borel map $\GeodSel\colon E\rightarrow \Geo(\Susp(X))$ such that $\GeodSel([x,t],[y,s])$ is a geodesic.
Once we have this, the rest of the proof follows as that of \cite[Theorem 2.10]{ambrosio-gigli}.
\end{proof}

The next result characterises geodesics joining $\delta_{\mathbf{0}}$ with measures supported on $X\times \lbrace \pi/2\rbrace.$  

\begin{lemma}\label{lemma.uniquegeodesicfromvertex}
Let $\mu \in \PP_p(\Susp(X))$ with $\supp(\mu) \subset X\times \lbrace \pi/2\rbrace$. Then there exists a unique geodesic joining $\delta_{\mathbf{0}}$ with $\mu$, given by
\[
s \mapsto \mu_s= (L_{s})_\#\mu, \quad 0\leq s\leq 1,
\]
where $L_{s}([x,t])= [x,st]$.
\end{lemma}

\begin{proof}
Let $\mu$ be a measure supported on $X\times\lbrace \pi/2\rbrace$, and $\widetilde{\mu}\in\PP(\Geo(\Susp(X)))$ given by $\widetilde{\mu}= f_\#\mu$, where the map $f\colon X\times \{\pi/2\}\to \Geo(\Susp(X))$ is defined by $f([x,t])(s)=L_s([x,t])$. A straightforward computation yields that 
\[
(L_s)_\#\mu = (e_s)_\#\widetilde{\mu},
\]
and, by lemma~\ref{teo.geodSusp}, $s\mapsto (L_s)_\#\mu$ is a geodesic between $\delta_{\mathbf{0}}$ and $\mu$. 

We now prove uniqueness. Let $(\nu_s)_{s\in [0,1]}$ and $(\eta_s)_{s\in [0,1]} $ be two geodesics joining $\delta_{\mathbf{0}}$ with $\mu$. By theorem~\ref{teo.geodSusp}, there exist measures $\widetilde{\nu},\widetilde{\eta}\in \PP_p(\Geo(\Susp(X)))$ such that  $e_{s\#}\widetilde{\nu}= \nu_s$ and  $e_{s\#}\widetilde{\eta}= \eta_s$  for all $0\leq s\leq 1$. In particular, $\nu_{s}$ and $\eta_{s}$ are supported on $X\times \lbrace s\pi/2\rbrace$, for any $s\in (0,1)$. Moreover, for any $x,y\in X$, we have 
\[
d_{\Susp}([x,s\pi/2],[y,\pi/2])\geq (1-s)\pi/2,
\]
with equality if and only if $x=y$. Since $\W_p(\nu_{s},\mu)=(1-s)\pi/2$, we have $([x,s\pi/2],[y,\pi/2])\in \supp(\pi_\nu)$ if and only if $x=y$, whenever $\pi_\nu \in \Opt(\nu_{s},\mu)$. This implies that $\pi_\nu$ is induced by a map $T^s_\nu$  given by $[x,s\pi/2]\mapsto[x,\pi/2]$ when restricted to $\supp(\nu_{s})$. An analogous argument yields the same result for optimal plans between $\eta_{s}$ and $\mu$. In particular, $T^s_{\nu\#}\nu_{s}=\mu=T^s_{\eta\#}\eta_{s}$, which implies that $\nu_{s}=\eta_{s}$.
\end{proof}

\subsection{Invariance of Dirac measures}
We now show that Dirac measures are invariant under the action of isometries of the Wasserstein space. We first consider the Dirac measures $\delta_{\mathbf{0}}$ and $\delta_{\pi}$.

\begin{lemma}\label{prop.invariancedeltazeropi}
Let $\Phi \in \Isom (\PP_p(\Susp(X))).$ Then $\Phi (\lbrace \delta_{\mathbf{0}}, \delta_{\ppi}\rbrace) = \lbrace \delta_{\mathbf{0}}, \delta_{\ppi}\rbrace.$     
\end{lemma}

\begin{proof}
    Let $\Phi \in \Isom (\PP_p(\Susp(X)))$ and notice that 
    \[\W_p^p (\Phi(\delta_{\mathbf{0}}), \Phi(\delta_{\ppi}))=\W_p^p (\delta_{\mathbf{0}},\delta_{\ppi})=\pi^p.\]
 Thus, by lemma \ref{lemma.diametersuspension}(3), we have $\{\Phi(\delta_{\mathbf{0}}),\Phi(\delta_{\ppi})\} = \lbrace \delta_{\mathbf{0}},\delta_{\ppi} \rbrace$.
\end{proof}

Using the preceding lemma, we show that the set of measures supported on $X\times \{\pi/2\}$ is invariant under isometries.

\begin{lemma}\label{prop.invarianceXtimespi2}
Let $\Phi\in \Isom(\PP_p(\Susp(X)))$. Then the following assertions hold:
\begin{enumerate}
\item $\Phi(\PP_p(X\times \lbrace \pi/2 \rbrace)) = \PP_p(X\times \lbrace \pi/2 \rbrace)$.
\item If $\Phi(\delta_{\textbf{0}})=\delta_{\textbf{0}}$, then $\Phi(\PP_p(X\times \lbrace t \rbrace)) = \PP_p(X\times \lbrace t \rbrace)$, for any $t\in[0,1]$. 
\item If $\iota\colon X\to X\times\{\pi/2\}\subset \Susp(X)$ is the isometric embedding $x\mapsto [x,\pi/2]$, then 
\[
\Phi |_{\PP_p(X\times \lbrace\pi/2\rbrace)} = \iota_\#\circ\widetilde{\Phi}\circ \iota^{-1}_\#
\]
for some $\widetilde{\Phi}\in \Isom(\PP_p(X))$.     
\end{enumerate}
\end{lemma}

\begin{proof}
Lemma \ref{teo.geodSusp} implies that
\[
\Mid(\delta_{\mathbf{0}},\delta_{\ppi}) = \PP_p(X\times\lbrace \pi/2\rbrace).
\]
Combined with lemma~\ref{prop.invariancedeltazeropi}, we obtain item (1).

Now fix $0<t<\pi/2$. Consider $\mu \in \PP_p(X\times \lbrace t \rbrace)$ and define $\widetilde\mu = L_{\frac{\pi}{2t}\#}\mu\in \PP_p(X\times\{\pi/2\})$, where $L_{\frac{\pi}{2t}}$ is defined as in lemma~\ref{lemma.uniquegeodesicfromvertex}. 
Therefore, $\mu$ is in the interior of the unique geodesic joining $\delta_{\mathbf{0}}$ with $\widetilde{\mu}.$ 
As $\Phi(\widetilde{\mu})$ is still supported on $X\times \lbrace\pi/2\rbrace$ by the first part of the lemma, and $\Phi(\mu)$ is in the interior of the unique geodesic joining $\delta_{\mathbf{0}}$ with $\Phi(\widetilde\mu)$, we conclude that $\Phi(\mu)$ is supported on $X\times \lbrace t \rbrace$. An analogous argument yields the case $\pi/2<t<\pi$. The case $t\in\{0,\pi\}$ follows from lemma \ref{prop.invariancedeltazeropi}.
\end{proof}

\begin{lemma}
Let $\Phi\in \Isom(\PP_p(\Susp(X)))$. Then the set $\lbrace (1-\lambda)\delta_{\mathbf{0}}+\lambda\delta_{\ppi}:\lambda \in [0,1] \rbrace$ is invariant under $\Phi$. Moreover, either $\Phi$ fixes $\lbrace (1-\lambda)\delta_{\mathbf{0}}+\lambda\delta_{\ppi}:\lambda \in [0,1] \rbrace$ pointwise, or 
\[
\Phi((1-\lambda)\delta_{\mathbf{0}}+\lambda\delta_{\ppi})=\lambda\delta_{\mathbf{0}}+(1-\lambda)\delta_{\ppi}
\]
for all $\lambda\in [0,1]$. In particular, the measure $\frac{1}{2}\delta_{\mathbf{0}}+\frac{1}{2}\delta_{\ppi}$ is fixed by $\Phi$.   
\end{lemma}

\begin{proof}
Let $\nu \in \PP_p(X\times \lbrace \pi/2 \rbrace)$. Observe that 
\[
\max\left\{\W_p^p(\nu,\mu):\mu\in\PP_p(\Susp(X))\right\} = (\pi/2)^p.
\]
Moreover,
\[\mu \in \argmax\left(\W_p^p(\nu,\cdot)\right)\]
if and only if $\supp(\mu) \subset \lbrace{\mathbf{0}}, {\ppi}\rbrace$. In other words, 
\[
\argmax\left(\W_p^p(\nu,\cdot)\right) = \{(1-\lambda)\delta_{\mathbf{0}}+\lambda \delta_{\ppi}:\lambda\in [0,1]\}.
\]
By lemma \ref{prop.invarianceXtimespi2}, we have  $\Phi (\nu)\in \PP_p(X\times \lbrace\pi/2\rbrace)$. Therefore, the set $\lbrace (1-\lambda)\delta_{\mathbf{0}}+\lambda\delta_{\ppi}: \lambda \in [0,1] \rbrace$ is invariant under $\Phi$.

Now, for the second part, fix $\lambda\in [0,1]$. Then
\begin{align*}
\lambda\pi^p &=
\W_p^p(\delta_{\mathbf{0}},(1-\lambda)\delta_{\mathbf{0}}+\lambda\delta_{\ppi})\\
&=
\W_p^p(\Phi(\delta_{\mathbf{0}}),\Phi((1-\lambda)\delta_{\mathbf{0}}+\lambda\delta_{\ppi}))\\
&= \W_p^p(\Phi(\delta_{\mathbf{0}}), (1-\lambda')\delta_{\mathbf{0}}+\lambda'\delta_{\ppi}).
\end{align*}
By lemma~\ref{prop.invariancedeltazeropi}, $\Phi(\delta_{\mathbf{0}})\in \lbrace\delta_{\mathbf{0}},\delta_{\ppi}\rbrace$. 
Assume first that $\Phi(\delta_{\mathbf{0}}) = \delta_{\mathbf{0}}$. Then $\lambda'=\lambda$. Suppose now that $\Phi(\delta_{\mathbf{0}})=\delta_{\ppi}$. In this case,  $\lambda'=1-\lambda$. Either way, $\lambda'=1/2$ when $\lambda =1/2$.
\end{proof}

The following lemma provides a metric characterisation of Dirac measures supported on the equator of the suspension.

\begin{lemma}\label{prop.characterizationDiracdeltaspi/2}
Let $\nu\in \PP_p(X\times\{\pi/2\})$. Then $\nu \in \Delta_1(X\times \lbrace \pi/2\rbrace)$ if and only if 
the set of midpoints $\Mid(\nu,(1-\lambda)\delta_{\mathbf{0}}+\lambda\delta_{\ppi})$ has exactly one element for any $\lambda\in(0,1)$.
\end{lemma}

\begin{proof}
First, consider the measure $\delta_{[x,\pi/2]}\in\Delta_1(X\times\{\pi/2\})$. Clearly,
\[
\Mid(\delta_{[x,\pi/2]},
    (1-\lambda)\delta_{\mathbf{0}}+\lambda\delta_{\ppi})
    = \lbrace 
    (1-\lambda)\delta_{[x,\pi/4]}+\lambda\delta_{[x,3\pi/4]}
    \rbrace
\]
and the ``only if'' implication follows. 

To prove the ``if'' implication, we will show that, if $\nu \in \PP_p(X\times \lbrace \pi/2 \rbrace)\setminus \Delta_1(X\times \lbrace \pi/2\rbrace)$, then $\Mid(\nu,(1-\lambda)\delta_{\mathbf{0}}+\lambda\delta_{\ppi})$ contains more than one element.

Let $E\subset X\times \lbrace \pi/2\rbrace$ be a Borel set such that $0<\nu (E)<1$. This set exists because, since $\nu$ is not a Dirac measure, $\supp(\nu)$ has at least two points, say $x$ and $y$. Then, we can take $\eps>0$ such that $B_\eps(x)\cap B_\eps(y) = \varnothing$, and define $E = B_{\eps/2}(x)\cap X\times\{\pi/2\}$. Let $F = X\times \{\pi/2\}\setminus E$, set $\lambda = \nu(F)$, and write $\nu$ as
\[
\nu = (1-\lambda)\frac{\nu\llcorner E}{\nu(E)}+\lambda \frac{\nu\llcorner F}{\nu(F)}.
\]
The midpoint measures
\[
\eta_0 \in \Mid\left(\delta_{\mathbf{0}},\frac{\nu\llcorner E}{\nu(E)}\right),\quad \eta_\pi \in \Mid\left(\frac{\nu\llcorner F}{\nu(F)},\delta_{\ppi}\right)
\]
 satisfy 
\[
\eta_0(L_{1/2}(E))=1, \quad \eta_{\pi}(L_{3/2}(F))=1,
\]
where the map $L_s\colon \Susp(X) \to \Susp(X)$ is given by $L_{s}([x,t])= [x,st]$, as in lemma \ref{lemma.uniquegeodesicfromvertex}. 

Computing the distances to $\nu$ and to 
$(1-\lambda)\delta_{\mathbf{0}}+\lambda\delta_{\ppi}$, 
we may verify that 
\[
(1-\lambda)\eta_0+\lambda\eta_{\pi}\in \Mid(\nu,(1-\lambda)\delta_{\mathbf{0}}+\lambda\delta_{\ppi}).
\]
Let us now show that there exists at least another midpoint between $\nu$ and 
$(1-\lambda)\delta_{\mathbf{0}}+\lambda\delta_{\ppi}$. 
Consider the measures
\[
\nu_0 \in \Mid(\delta_{\mathbf{0}},\nu),\quad \nu_\pi\in \Mid(\nu,\delta_{\ppi}).
\]
Computing the distances to $\nu$ and 
$(1-\lambda)\delta_{\mathbf{0}}+\lambda\delta_{\ppi}$, 
we obtain
\[(1-\lambda)\nu_0+\lambda\nu_{{\ppi}}\in \Mid(\nu,(1-\lambda)\delta_{\mathbf{0}}+\lambda\delta_{\ppi}). \]
To finish the argument, 
 notice that both measures $(1-\lambda)\eta_0+\lambda\eta_{\pi}$ and $(1-\lambda)\nu_0+\lambda\nu_{\pi}$ are in $I(\nu)$, and observe that
\[
\left((1-\lambda)\nu_0+\lambda\nu_{\pi}\right)(L_{1/2}(E))= (1-\lambda)^2<(1-\lambda)=\left((1-\lambda)\eta_0+\lambda\eta_{\pi}\right)(L_{1/2}(E)).
\]
Therefore, the midpoints we obtained are indeed different measures.
\end{proof}

The previous lemma allows us to prove that the set of Dirac measures is invariant.

\begin{proposition}\label{prop.invarianceDiracdeltas}
 If $\Phi \in \Isom(\PP_p(\Susp(X))),$ then $\Phi(\Delta_1(\Susp(X)))=\Delta_1(\Susp(X)). $   
\end{proposition}

\begin{proof}
    First, consider the measure $\delta_{[x,\pi/2]}$ for some $x\in X$. By lemma \ref{prop.characterizationDiracdeltaspi/2}, the set of midpoints $\Mid\left(\delta_{[x,\pi/2]},(1-\lambda)\delta_{\mathbf{0}}+\lambda\delta_{\ppi}\right)$
    has exactly one element for any $\lambda \in [0,1]$. Fix $\Phi\in\Isom(\PP_p(\Susp(X)))$. Then
    \[
    \Phi\left(\Mid\left(\delta_{[x,\pi/2]},(1-\lambda)\delta_{\mathbf{0}}+\lambda\delta_{\ppi}\right)\right)=\Mid\left(\Phi(\delta_{[x,\pi/2]}),(1-\lambda')\delta_{\mathbf{0}}+\lambda'\delta_{\ppi}\right),
    \]
    where either $\lambda'=\lambda$ or $\lambda'=1-\lambda$.
    Hence, the set of midpoints 
    $\Mid(\Phi(\delta_{[x,\pi/2]}),(1-\lambda')\delta_{\mathbf{0}}+\lambda'\delta_{\ppi})= \lbrace \ast\rbrace$ also has exactly one element for any $\lambda'\in [0,1]$. Therefore, by lemma~\ref{prop.characterizationDiracdeltaspi/2}, $\Phi(\delta_{[x,\pi/2]})\in \Delta_1(X\times \lbrace\pi/2\rbrace)$.
    
Consider now $\delta_{[x,s]}$ with $s \not\in \{0,\pi/2,\pi\}$. Observe that there exists only one geodesic joining $\mathbf{0}$ and $\ppi$ passing through $[x,s]$, which implies there is only one geodesic joining $\delta_{\mathbf{0}}$ with $\delta_{\ppi}$ and passing through $\delta_{[x,s]}$ by lemma~\ref{teo.geodSusp}. This geodesic is given by
    \[
    \gamma\colon t \mapsto \delta_{[x,t\pi]}, \quad t \in [0,1].
    \]
    As $\Phi(\delta_{[x,\pi/2]})=\delta_{[y,\pi/2]}$ for some $y \in X$, the image of the geodesic $\gamma$ under $\Phi$  is the geodesic given by
    \[
    \Phi(\gamma)\colon t \mapsto \delta_{[y,t\pi]}, \quad t \in [0,1],
    \]
    which implies that $\Phi(\delta_{[x,s]})\in \Delta_1(\Susp(X))$.

Finally, for $\delta_{[x,s]}$ with $s\in\{0,\pi\}$, we have $\Phi(\delta_{[x,s]})\in\Delta_1(\Susp(X))$ by lemma~\ref{prop.invariancedeltazeropi}.
\end{proof}

\subsection{Invariance of measures on meridians}
Having proved that the set of Dirac measures on $\Susp(X)$ is invariant under the action of $\Isom(\PP_p(\Susp(X)))$, proving the isometric rigidity of $\Susp(X)$ reduces to verifying that the group 
\[
K = \lbrace \Phi \in \Isom (\PP_p(\Susp(X))) : \Phi (\delta_{[x,t]})= \delta_{[x,t]} \text{ for all } [x,t]\in \Susp(X) \rbrace.
\]
is trivial (see proposition~\ref{lem:characterisation of isometric rigidity}). In the rest of the section, we will show that this is the case, proving Theorem~\ref{t:rigidity of compact spaces}. 
First, we show that $K$ fixes measures supported on geodesics joining $\mathbf{0}$ and $\ppi$. We will refer to such geodesics as \textit{meridians}. To this end, we need the following lemmas. 

\begin{lemma}\label{lem.distanceproj}
Let $0\leq t \leq \pi.$ Then the map $T_t\colon \Susp(X)\setminus\lbrace {\mathbf{0}},{\ppi}\rbrace \rightarrow X\times \lbrace t \rbrace$ given by $T_t([x,s])=[x,t]$ is continuous and if we define it arbitrarily on $\lbrace {\mathbf{0}},{\ppi}\rbrace$, the resulting extension of $T_t$ to the whole suspension $\Susp(X)$ is a Borel map.
Moreover, $T_t$ satisfies
\[
\argmin(d_{\Susp}([x,s],\cdot)|_{X\times\{t\}}) = \{T_{t}([x,s])\}
\]
for all $[x,s]\in \Susp(X)$.
\end{lemma}

\begin{proof}
The continuity of $T_t$ follows from the fact that the quotient map $X\times (0,\pi)\to \Susp(X)\setminus\{\mathbf{0},\ppi\}$ is a homeomorphism, and the map $T_t$ is given by the projection $(x,s)\mapsto (x,t)$ under this homeomorphism whenever $t\in(0,\pi)$. On the other hand, when $t \in \{0,\pi\}$, the map $T_t$ is  constant.

To verify the second assertion, let $\widetilde{T}_t$ be an extension of $T_t$ to all of $\Susp(X).$ Take a Borel subset $A\subset X\times \lbrace t\rbrace$ and observe that $\widetilde{T}^{-1}_{t}(A)$ consists of $T^{-1}_t(A)$ with possibly the union with $\lbrace {\mathbf{0}}\rbrace, \lbrace {\ppi}\rbrace$ or $\lbrace {\mathbf{0}},{\ppi}\rbrace$ depending on whether $\widetilde{T}_t({\mathbf{0}})$ or$\widetilde{T}_t({\ppi})$ belong to $A$. Since $T_t$ is continuous, $T^{-1}_t(A)$ is a Borel set and, as the sets $\lbrace {\mathbf{0}}\rbrace, \lbrace {\ppi}\rbrace$ and $\lbrace {\mathbf{0}},{\ppi}\rbrace$ are closed, $\widetilde{T}^{-1}_t(A)$ is a Borel set.

Finally, if $[y,t]\in X\times\{t\}$ then 
\[
d_{\Susp}([x,s],[y,t]) \geq |s-t| = d_{\Susp}([x,s],T_t([x,s])) 
\]
with equality if and only if $x=y$, and the last claim follows.
\end{proof}

\begin{corollary}\label{cor:non-branching.geodesics.starting.at.pole}
Geodesics starting at $\mathbf{0}$ or $\ppi$ are non-branching. In other words, if $\gamma$ and $\gamma'$ are geodesics in $\Susp(X)$ such that $\gamma_0 = \gamma'_0 = \mathbf{0}$ and $\gamma_t = \gamma'_t$ for some $t\in (0,1)$, then $\gamma_1=\gamma'_1$. 
\end{corollary}
\begin{proof}
If $\gamma_t=\gamma'_t = \mathbf{0}$ for some $t\in (0,1)$ then $\gamma$ and $\gamma'$ are the constant geodesic identical to $\mathbf{0}$. On the other hand, if $\gamma_t = \gamma'_t \in X\times \{s\}$ for some $t\in (0,1)$ and $s>0$, then $\gamma_1,\gamma'_1\in X\times \{s/t\}$. By lemma~\ref{lem.distanceproj}, we have 
\[
\{\gamma_1\} = \argmin(\W_p(\gamma_t,\cdot)|_{\PP_p(X\times \{s/t\})}) = \argmin(\W_p(\gamma'_t,\cdot)|_{\PP_p(X\times \{s/t\})}) = \{\gamma'_1\},
\]
and the claim follows.
\end{proof}

\begin{lemma}\label{lem.projectionintospheres}
Let $0<t<\pi$ and fix a measure $\mu \in \PP_p(\Susp(X))$.
\begin{enumerate}
\item If $\mu(\lbrace {\mathbf{0}},{\ppi} \rbrace)=1$, then 
\[
\W_p(\mu,\nu)=\W_p(\mu,\PP_p(X\times\lbrace t\rbrace)) 
\]
for all $\nu \in \PP_p(X\times \lbrace t\rbrace)$.
\item  For any extension of $T_t\colon \Susp(X)\setminus\{\mathbf{0},\ppi\}\to X\times \{t\}$ as in lemma \ref{lem.distanceproj}, which we still denote by $T_t$ for ease of notation, we have
\[
\W_p(\mu,T_{t\#}\mu)= \W_p(\mu,\PP_p(X\times \lbrace t\rbrace)).
\]
\item The measure $\mu$ is supported on the image of a geodesic $\gamma$ with $\gamma_0={\mathbf{0}}, \gamma_1={\ppi}$ and $\gamma_t = [x,t]$ if and only if \[\W_p(\mu,\delta_{[x,t]})=\W_p(\mu,\PP_p(X\times\lbrace t\rbrace)).\]
\end{enumerate}
\end{lemma}

\begin{proof}
We first prove assertion (1). Observe that $\mu=(1-\lambda)\delta_{\mathbf{0}}+\lambda\delta_{\ppi}$ for some $\lambda \in [0,1].$ Fix $\nu\in \PP_p(X\times\lbrace t\rbrace)$. We will show that the plan  $\mu\otimes \nu$ is optimal. Recall that $d_{\Susp}({\mathbf{0}},[x,t])=t$ and that $d_{\Susp}({\ppi},[x,t])=\pi-t$ for any $x\in X$, which implies that the support of $\mu\otimes\nu$ is cyclically monotone. Therefore, 
\begin{align*}
\W^{p}_p(\mu,\nu) &= \int d^{p}_{\Susp(X)}([y,s],[x,t])\ d(\mu\otimes\nu)([y,s],[x,t])\\
&= (1-\lambda)t^p+\lambda(\pi-t)^p.
\end{align*}
It follows that $\W_p(\mu,\nu)$ is constant with respect to $\nu\in \PP_p(X\times \{t\})$, proving (1).

Now, we prove assertion (2). Consider an arbitrary $\nu\in \PP_p(X\times\{t\})$ and let $\pi\in \Opt(\mu,\nu)$. Then
\begin{align*}
    \W_p^p(\mu,\nu) &= \int d^p_{\Susp}([y,s],[x,t])\ d\pi([y,s],[x,t])\\
    &\geq \int |s-t|^p\ d\pi([y,s],[x,t])\\
    &= \int d^p_{\Susp}([y,s],T_t[y,s])\ d\mu([y,s])\\
    &\geq \W_p^p(\mu,T_{t\#}\mu). 
\end{align*}
Since $T_t$ is a Borel map, $T_{t\#}\mu\in\PP_p(X\times\{t\})$, which proves (2). Observe that, setting $\nu = T_{t\#}\mu$ in the preceding inequalities, we may conclude that $T_t$ is an optimal map between $\mu$ and $T_{t\#}\mu$.

Finally, we prove assertion (3). We first verify the ``if'' implication. Let $x\in X$ be such that $\W_p(\mu,\delta_{[x,t]})=\W_p(\mu,\PP_p(X\times\lbrace t\rbrace))$. Suppose first that $\mu(\{\mathbf{0},\ppi\})=1$. Then $\mu = (1-\lambda)\delta_{\mathbf{0}}+\lambda\delta_{\ppi}$ for some $\lambda \in [0,1]$, and the conclusion follows. Assume now that $0<\mu(\Susp(X)\setminus \lbrace {\mathbf{0}},{\ppi} \rbrace)\leq 1$. Then, by item (2), 
\[
\W_p^p(\mu,\delta_{[x,t]})=\W_p^p\left(\mu,T_{t\#}\mu\right).
\]
Therefore,
\begin{align*}
\int d_{\Susp}^p([y,s],[x,t])d\mu([y,s])&=\int d_{\Susp}^p([y,s],T_t([y,s]))\ d\mu([y,s])\\
&= \int d_{\Susp}^p([y,s],[y,t])\ d\mu([y,s]),
\end{align*}
which implies that 
\[
\int d_{\Susp}^p([y,s],[x,t])-d_{\Susp}^p([y,s],[y,t])\ d\mu([y,s])=0.
\]
Since $d_{\Susp}^p([y,s],[x,t])-d_{\Susp}^p([y,s],[y,t])\geq 0$ for all $[y,s]$, we have 
\[
d_{\Susp}^p([y,s],[y,t])=d_{\Susp}^p([y,s],[x,t])\quad 
\mu\text{-a.e.}
\]
Hence, $[y,t]=[x,t]$ for $\mu$-a.e.\ $[y,s]\in \Susp(X)\setminus\{{\mathbf{0}},{\ppi}\}$. Therefore, all the points in $\supp(\mu)$ are of the form $[x,s]$ and the conclusion follows.

The ``only if'' implication follows by applying item (2) with the extension of $T_t$ given by $T_t(\mathbf{0}) = T_t(\ppi) = [x,t]$. 
\end{proof}

The following proposition shows that measures supported on meridians are fixed by isometries $K$.

\begin{proposition}
\label{thm.fixinggeodesics}
 Let $\gamma \in \Geo(\Susp(X))$ with $\gamma_0={\mathbf{0}},$ $\gamma_1 = {\ppi}$  
 and fix an isometry $\Phi \in K$. Then $\Phi(\mu) = \mu$ for all $\mu \in \PP_p(\gamma)$, i.e.\ $\supp(\mu) \subset \gamma([0,1])$.
\end{proposition}

\begin{proof}
Let $\mu\in \PP_p(\gamma).$ By lemma \ref{lem.projectionintospheres},  
$\W_p(\mu,\delta_{\gamma_{1/2}})=\W_p(\mu,\PP_p(X\times\{\pi/2\}))$. Since $\Phi (\delta_{\gamma_{1/2}})=\delta_{\gamma_{1/2}}$ and $\PP_p(X\times\{\pi/2\})$ is invariant under $\Phi$, we have 
\[
\W_p(\Phi(\mu),\delta_{\gamma_{1/2}})=\W_p(\Phi(\mu),\PP_p(X\times\{\pi/2\})).
\]
Therefore, by item (2) in lemma \ref{lem.projectionintospheres}, $\Phi(\mu)$ is supported on $\gamma([0,1])$. Thus, $\Phi |_{\PP_p(\gamma)}\in \Isom(\PP_p(\gamma))$. By \cite[Lemma~4.4]{Santos2022}, a closed interval is isometrically rigid (although the proof of this lemma is only stated for $p=2$, the argument holds for all $p>1$). Since $\gamma([0,1])$ is isometric to a closed interval, we conclude that $\Phi(\mu)=\mu$.
\end{proof}

\subsection{Invariance of measures on parallels}
Now, we will prove that measures supported on parallels, that is, sets of the form $X\times\{t\}$, are invariant under the action of $\Isom(\PP_p(\Susp(X)))$.  To this end, we consider projections onto half-meridian geodesics, i.e.\ geodesics joining one pole of the suspension to the equator.

\begin{lemma}\label{lemma.projectionpointsintogeodesics}
 Fix $0\leq t \leq \pi/2$, $[x,t]\in \Susp(X)$   and let $\gamma$ be a geodesic with 
 $\gamma_0 ={\mathbf{0}}$ and $\gamma_1 \in X\times \lbrace \pi/2\rbrace$. Then the projection map $\proj_\gamma\colon \Susp(X)\to \gamma([0,1])\subset \Susp(X)$ given by  
\[
\proj_{\gamma}([x,t]) = \argmin \left(d_{\Susp}([x,t],\cdot)|_{\gamma([0,1])}\right)
\] 
is continuous. Moreover, if $t<\pi/2$, then $\proj_{\gamma}([x,t]) = \gamma_{s}$ with $\frac{\pi}{2} s=t$ if and only if $[x,t]=\gamma_{s}.$ 
\end{lemma}

\begin{proof}
Fix $[x,t]\in \Susp(X)$ with $t\in [0,\pi/2]$ and let $\gamma$ be a geodesic given by
\[
\gamma_s = [y,s\pi/2], \quad 0\leq s\leq 1
\]
for some $y\in X$. Then
\begin{align*}
d_{\Susp}([x,t], \gamma_s) & = d_{\Susp}([x,t], [y,s\pi/2]) \\ 
& =\cos^{-1}\left(\cos(t)\cos(s\pi/2)+\sin(t)\sin(s\pi/2)\cos(d(x,y))\right).   
\end{align*}
It is then clear that the function $s\mapsto d_{\Susp}([x,t], \gamma_s)$ is differentiable. Taking into account that $d(x,y)< \pi/2$, we may compute the (unique) minimum explicitly:
\[
\gamma_s=
\begin{cases}
[y, \tan^{-1}\left(\cos(d(x,y)\tan(t)\right)], &0\leq t<\pi/2;\\
[y,\pi/2], &t=\pi/2.
\end{cases}
\]
Therefore, the map $[x,t] \mapsto \proj_{\gamma}([x,t]) = \gamma_s$
is well-defined and continuous. The second statement follows from the explicit description of $\proj_{\gamma}([x,t])$.
\end{proof}

As a consequence of the previous lemma, we get an analogous result for probability measures.

\begin{lemma}\label{lemma.projectionintogeodics}
Fix $0<t<\pi/2$ and let $\gamma$ be a geodesic with $\gamma_0 ={\mathbf{0}}$ and $\gamma_1 \in X\times \lbrace \pi/2\rbrace$.  If $\mu \in \PP_p(X \times \lbrace t \rbrace)$, then the following assertions hold:
\begin{enumerate} 
\item $\argmin\left(\W_p(\cdot,\mu)|_{\PP_p(\gamma)}\right) = \{\proj_{\gamma \#}\mu\}$. 
\item $\gamma_t$ is an atom of $\mu$ if and only if it is an atom of $\proj_{\gamma\#}\mu$. More precisely, 
\[\mu (\{\gamma_t\}) = \proj_{\gamma\#}\mu(\{\gamma_t\}).\]
\end{enumerate}
\end{lemma}

\begin{proof} First, we prove (1). 
Observe that $\proj_{\gamma\#}\mu\in \argmin(\W_p(\mu,\cdot)|_{\PP_p(\gamma)})$. Indeed, for any $\nu\in \PP_p(\gamma)$ and any $\pi\in\Opt(\mu,\nu)$ we have 
\begin{align*}
\W_p^p(\mu,\nu) &= \int d^p_{\Susp}([x,t],[y,s])\ d\pi([x,t],[y,s])\\
&\geq \int d^p_{\Susp}([x,t],\proj_\gamma([x,t]))\ d\mu([x,t])\\
&\geq \W^p_p(\mu,\proj_{\gamma\#}\mu).
\end{align*}
Here, the first inequality follows from the definition of the map $\proj_\gamma$, and the second one follows from the definition of $\W_p$ and the fact that $\proj_\gamma$ is continuous (see lemma~\ref{lemma.projectionpointsintogeodesics}). Furthermore, if $\eta \in \argmin(\W_p(\mu,\cdot)|_{\PP_p(\gamma)})$ then the previous inequalities are all equalities, and in particular,
\[
\int d^p_{\Susp}([x,t],[y,s])\ d\pi([x,t],[y,s]) = \int d^p_{\Susp}([x,t],\proj_\gamma([x,t]))\ d\mu([x,t]).
\]
It follows that 
\[
d_{\Susp}([x,t],[y,s]) = d_{\Susp}([x,t],\proj_\gamma([x,t])) \quad \text{$\pi$-a.e. $([x,t],[y,s])$}
\]
which in turn implies that $\pi = (\id,\proj_\gamma)_\#\mu$. Since $\pi$ is an optimal plan between $\mu$ and $\eta$, we conclude that $\eta = \proj_{\gamma\#}\mu$.

Finally, by lemma~\ref{lemma.projectionpointsintogeodesics} and the fact that $t <\pi/2$ and $\mu\in\PP_p(X\times\{t\})$, we have
\[
\proj_{\gamma\#}\mu (\{\gamma_t\})= \mu (\proj_{\gamma}^{-1}(\{\gamma_t\})\cap X\times\{t\}) = \mu(\{\gamma_t\}),
\]
and item (2) follows.
\end{proof}

\begin{lemma}\label{lem:invariance.atomic.measures.equator}
If $\mu \in \Delta_n(X\times \lbrace \pi/2\rbrace)$  and $\Phi \in K$, then $\Phi(\mu) = \mu$.
\end{lemma}

\begin{proof}
Let $\mu = \sum_{i=1}^n \lambda_i\delta_{[x_i,\pi/2]}$ be an atomic measure supported on $X\times \lbrace\pi/2 \rbrace.$ There exists a unique geodesic joining $\delta_{\mathbf{0}}$ with $\mu$. Let $\mu_{1/2}$ be the midpoint of this geodesic, and observe that $\supp(\mu_{1/2})\subset X\times \lbrace\pi/4\rbrace.$
Denote by $\gamma_i$ the unique geodesic in $\Susp(X)$ joining ${\mathbf{0}}$ with $[x_i,\pi/2].$  We will now consider the projection of the measure $\mu_{1/2}$ to each $\PP_p(\gamma_i)$. By theorem~\ref{thm.fixinggeodesics}, the isometry $\Phi$ restricted to $\PP_p(\gamma_i)$ is the identity. Therefore, 
\[
\W_p(\mu_{1/2},\PP_p(\gamma_i)) =\W_p(\Phi(\mu_{1/2}),\PP_p(\gamma_i)),
\]
which implies
\begin{align*}
\{\proj_{\gamma_i\#}\mu_{1/2}\}&=\argmin\left(\W_p(\mu_{1/2},\cdot)|_{\PP_p(\gamma_i)}\right)\\
&=\argmin\left(\W_p(\Phi(\mu_{1/2}),\cdot)|_{\PP_p(\gamma_i)}\right)\\
&=\{\proj_{\gamma_i\#}\Phi(\mu_{1/2})\}.
\end{align*}
Hence, $\mu_{1/2}$ and $\Phi(\mu_{1/2})$ have the same atoms with the same weights, by lemma~\ref{lemma.projectionintogeodics}. In other words, $\Phi(\mu_{1/2})=\mu_{1/2}$ and, by corollary~\ref{cor:non-branching.geodesics.starting.at.pole}, $\Phi(\mu)=\mu$.
\end{proof}

\begin{proposition}\label{lemma.fixslice}
If $t \in [0,\pi]$ and $\Phi \in K,$ then $\Phi(\mu)=\mu$ for all $\mu \in \PP_p(X\times \lbrace t \rbrace).$
\end{proposition}
\begin{proof}
The result for $t\in\{0,\pi\}$ holds trivially by the very definition of $K$. For $t = \pi/2$, the result follows from lemma~\ref{lem:invariance.atomic.measures.equator} and the fact that atomic measures are dense in $\PP_p(X\times\{\pi/2\})$ by the compactness of $X$. Finally, for arbitrary $t\in(0,\pi/2)$ and $\mu\in\PP_p(X\times\{t\})$, there is $\widetilde\mu\in \PP_p(X\times \{\pi/2\})$ with $[x,\pi/2]\in\supp(\widetilde\mu)$ if and only if $[x,t]\in\supp(\mu)$. Moreover, $\mu$ is in the interior of a geodesic joining $\delta_{\mathbf{0}}$ and $\widetilde\mu$. Since $\Phi(\widetilde\mu) = \widetilde\mu$ and $\Phi(\delta_{\mathbf{0}})=\delta_{\mathbf{0}}$, it follows that $\Phi(\mu)=\mu$ by corollary~\ref{cor:non-branching.geodesics.starting.at.pole}. A similar argument yields the result for $t\in(\pi/2,\pi)$.
\end{proof}

\subsection{Isometric rigidity of suspensions.} In this final subsection we show that atomic measures are invariant under $\Isom(\PP_p(\Susp(X)))$ and, appealing to the density of these measures, we obtain theorem~\ref{t:rigidity of compact spaces}. We follow a strategy analogous to the one used for compact rank one symmetric spaces in \cite{Santos2022}. Since we are working in more generality, we will assume the following technical condition.

\begin{cond}\label{cond:general position}
    For any $\{x_1,\dots,x_k\}\subset X$ and $\{t_1,\dots,t_n\}\subset (0,\pi/2)\cup(\pi/2,\pi)$ such that $x_l\neq x_m$ and $t_i\neq t_j$ for $l\neq m$, $i\neq j$, there exists $x_o\in X$ such that all the numbers 
    \[
    \tan(t_j)\cos(d(x_o,x_m)) \quad (\text{with}\ 1\leq m\leq k,\ 1\leq j\leq n)
    \]
    are different.
\end{cond}

\begin{remark}\label{rem:manifolds and condition b}
Condition \ref{cond:general position} holds when $X$ is a finite-dimensional closed Riemannian manifold or Alexandrov space (with curvature bounded below) with $\diam(X) < \pi/2$. Indeed, let 
\[
S_{ijlm} = \left\{x\in X : \tan(t_i)\cos(d(x,x_l))=\tan(t_j)\cos(d(x,x_m))\right\} = f_{lm}^{-1}(c_{ij}),
\]
where 
\begin{equation}\label{eq:flm}
f_{lm}(x) = \frac{\cos(d(x,x_l))}{\cos(d(x,x_m))}
\end{equation}
and $c_{ij} = \frac{\tan(t_j)}{\tan(t_i)}$. In particular, Condition \ref{cond:general position} holds for some $x_o\in X$ if and only if 
\[
x_o \in \bigcap_{\substack{i\neq j\\ l\neq m}} X\setminus S_{ijlm}.
\] 

When $X$ is a Riemannian manifold, each $f_{lm}$ is smooth outside $\Cut(x_l)\cup \Cut(x_m) \cup \{x_l,x_m\}$, and a direct computation shows that $x$ is a critical point of $f_{lm}$ only if $d(x,x_l)=d(x,x_m)$. In particular, $f_{lm}(x) = 1$ for any critical point $x$ of $f_{lm}$. Since $c_{ij} \neq 1$ for any $i\neq j$, each $c_{ij}$ is a regular value of $f_{lm}$, which in turn implies that $S_{ijlm}$ is either empty or a codimension 1 submanifold in $X$, for every choice of $i,j,l,m$ with $i\neq j$, $l\neq m$. Therefore, $X\setminus S_{ijlm}$ is open and dense for any choice of $i\neq j$ and $l\neq m$ and, since the finite intersection of open and dense sets is open and dense itself, the claim follows.

When $X$ is an Alexandrov space, $f_{lm}$ defined as in \eqref{eq:flm} is semiconcave away from $\Cut(x_l)\cup \Cut(x_m)\cup\{x_l,x_m\}$, which in turn is a set of $n$-dimensional Hausdorff measure zero (see \cite[Proposition 3.1]{OtsuShioya1994}). Moreover, a direct computation using the first variation formula (see \cite[Theorem 3.5]{OtsuShioya1994}) proves that $f_{lm}=1$ at any critical point. Finally, the Morse lemma (see \cite{Perelman1993}) applied to regular points of $f_{lm}$ proves that $X\setminus S_{ijlm}$ is open and dense for any choice of $i\neq j$ and $l\neq m$. Therefore, the set
\[
\left(\bigcap_{\substack{i\neq j\\l\neq m}}X\setminus S_{ijlm}\right) \cap \left(\bigcap_{l\neq m} X\setminus(\Cut(x_l)\cup \Cut(x_m)\cup \{x_l,x_m\})\right),
\]
being a finite intersection of dense sets, is dense itself, and in particular it is non-empty.
\end{remark}

\begin{proposition}\label{prop:invariance of atomic measures in general position}
Let $n\in\NN$ and set
\begin{equation}\label{eq:atomic measure}
\mu = \sum_{j=1}^{n}\sum_{i=1}^{k_j}\lambda_{ij} \delta_{[x_{ij},t_j]} \in \PP_p(\Susp(X))
\end{equation}
with $0<t_1< \dots< t_n<\pi$ and $x_{i_1j_1}\neq x_{i_2j_2}$ for any $(i_1,j_1)\neq (i_2,j_2)$. Then, for any $\Phi\in K$, $\Phi(\mu)=\mu$.
\end{proposition}

\begin{proof}
We will proceed by induction over $n$. The result holds for $n=1$ by lemma \ref{lemma.fixslice}. Suppose now that $n\geq 2$ and the result holds for $n-1$. We will prove that it is also true for $n$. To do this, consider $\mu$ of the form \eqref{eq:atomic measure}, and assume that $t_j\neq \pi/2$ for all $j\in\{1,\dots,n\}$. Set
\[
\nu = \left(\sum_{i=1}^{k_1} \lambda_{i1}\right)\delta_{{\mathbf{0}}} + \left(\sum_{j=2}^{n}\sum_{i=1}^{k_j} \lambda_{ij}\right)\delta_{{\ppi}}.
\]
Hence, there exist $0<\widetilde{t}_1<\dots<\widetilde{t}_{n-1}<\pi$ and an atomic measure $\widetilde{\mu}$ supported on $\bigcup_{j=1}^{n-1} X\times \{\widetilde{t}_j\}$ such that $\mu\in\Int(\widetilde{\mu},\nu)$, i.e.\ $\mu$ is in the interior of a geodesic joining $\widetilde{\mu}$ and $\nu$ (see subsection~\ref{ss:geodesics and curvature bounds}). Indeed, we may define $\tilde{\mu}$ explicitly by letting 
\[
\widetilde{t}_j = \begin{cases}
    \frac{\pi t_1}{\pi+t_1-t_2}, & j = 1,\\
    \frac{\pi(t_{j+1}+t_1-t_2)}{\pi+t_1-t_2}, & 2\leq j\leq n-1,
\end{cases}
\]
\[
\widetilde{k}_j = \begin{cases}
    k_1 + k_2, & j = 1,\\
    k_{j+1}, & 2\leq j \leq n-1,
\end{cases}
\]
\[
\widetilde{\lambda}_{ij} = \begin{cases}
    \lambda_{i1} & j = 1,\ 1\leq i\leq k_1,\\
    \lambda_{(i-k_1)2} & j = 1,\ k_1+1\leq i\leq \widetilde{k}_1,\\
    \lambda_{i(j+1)} & 2\leq j \leq n-1,\ 1\leq i\leq \widetilde{k}_j,
\end{cases}
\]
\[
\widetilde{x}_{ij} = \begin{cases}
x_{i1}, & j=1,\ 1\leq i\leq k_1,\\
x_{(i-k_1)2}, & j=1,\ k_1+1\leq i\leq \widetilde{k}_1,\\
x_{i(j+1)}, & 2\leq j\leq n-1, 1\leq i\leq \widetilde{k}_j,
\end{cases}
\]
and setting
\[
\widetilde{\mu} = \sum_{j=1}^{n-1}\sum_{i=1}^{\widetilde{k}_j}\widetilde{\lambda}_{ij} \delta_{[\widetilde{x}_{ij},\widetilde{t}_j]}.
\]
Since $\nu$ is supported on $\{{\mathbf{0}},{\ppi}\}$,  proposition~\ref{thm.fixinggeodesics} implies that $\Phi(\nu)=\nu$. By the induction hypothesis, $\Phi(\widetilde{\mu})=\widetilde{\mu}$. Therefore, $\Phi(\mu)\in \Int(\widetilde{\mu},\nu)$.

By cyclical monotonicity, $\mu'\in \Int(\widetilde{\mu},\nu)$ if and only if 
\begin{equation}\label{eq:intermediate-atomic-measure}
    \mu' = \sum_{i=1}^{k_1} (\lambda'_{i1}\delta_{[x_{i1},t_1]} + \lambda''_{i1}\delta_{[x_{i1},t_2]}) + \sum_{i=1}^{k_2} (\lambda'_{i2}\delta_{[x_{i2},t_1]} + \lambda''_{i2}\delta_{[x_{i2},t_2]}) + \sum_{j=3}^{n}\sum_{i=1}^{k_j} \lambda_{ij}\delta_{[x_{ij},t_j]}
\end{equation}
with $\lambda'_{i1}+\lambda''_{i1}=\lambda_{i1}$ and $\lambda'_{i2}+\lambda''_{i2}=\lambda_{i2}$ (see Figure \ref{fig:rigidity of suspensions}).

\begin{figure}
    \centering
    \def\svgwidth{0.9\linewidth}
    %% Creator: Inkscape 1.3.2 (091e20e, 2023-11-25, custom), www.inkscape.org
%% PDF/EPS/PS + LaTeX output extension by Johan Engelen, 2010
%% Accompanies image file 'rigidity_of_suspensions.pdf' (pdf, eps, ps)
%%
%% To include the image in your LaTeX document, write
%%   \input{<filename>.pdf_tex}
%%  instead of
%%   \includegraphics{<filename>.pdf}
%% To scale the image, write
%%   \def\svgwidth{<desired width>}
%%   \input{<filename>.pdf_tex}
%%  instead of
%%   \includegraphics[width=<desired width>]{<filename>.pdf}
%%
%% Images with a different path to the parent latex file can
%% be accessed with the `import' package (which may need to be
%% installed) using
%%   \usepackage{import}
%% in the preamble, and then including the image with
%%   \import{<path to file>}{<filename>.pdf_tex}
%% Alternatively, one can specify
%%   \graphicspath{{<path to file>/}}
%% 
%% For more information, please see info/svg-inkscape on CTAN:
%%   http://tug.ctan.org/tex-archive/info/svg-inkscape
%%
\begingroup%
  \makeatletter%
  \providecommand\color[2][]{%
    \errmessage{(Inkscape) Color is used for the text in Inkscape, but the package 'color.sty' is not loaded}%
    \renewcommand\color[2][]{}%
  }%
  \providecommand\transparent[1]{%
    \errmessage{(Inkscape) Transparency is used (non-zero) for the text in Inkscape, but the package 'transparent.sty' is not loaded}%
    \renewcommand\transparent[1]{}%
  }%
  \providecommand\rotatebox[2]{#2}%
  \newcommand*\fsize{\dimexpr\f@size pt\relax}%
  \newcommand*\lineheight[1]{\fontsize{\fsize}{#1\fsize}\selectfont}%
  \ifx\svgwidth\undefined%
    \setlength{\unitlength}{499.94998962bp}%
    \ifx\svgscale\undefined%
      \relax%
    \else%
      \setlength{\unitlength}{\unitlength * \real{\svgscale}}%
    \fi%
  \else%
    \setlength{\unitlength}{\svgwidth}%
  \fi%
  \global\let\svgwidth\undefined%
  \global\let\svgscale\undefined%
  \makeatother%
  \begin{picture}(1,0.36196608)%
    \lineheight{1}%
    \setlength\tabcolsep{0pt}%
    \put(0,0){\includegraphics[width=\unitlength,page=1]{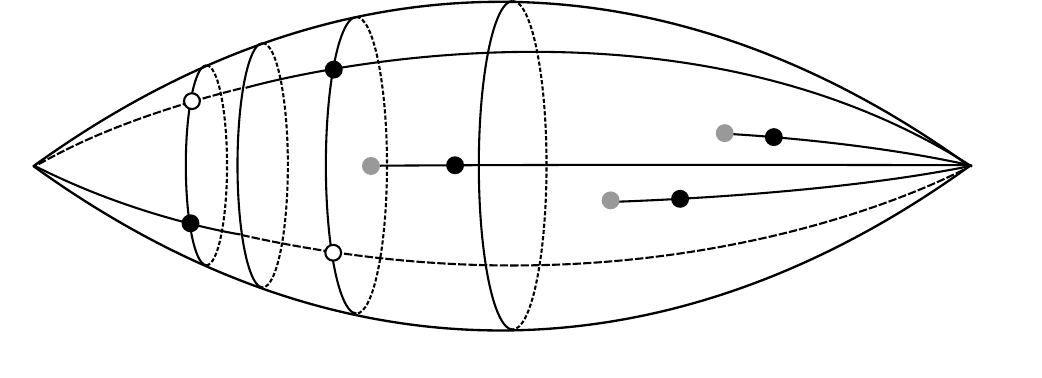}}%
    \put(-0.00036171,0.19831596){\color[rgb]{0,0,0}\makebox(0,0)[lt]{\lineheight{1.25}\smash{\begin{tabular}[t]{l}$\bf 0$\end{tabular}}}}%
    \put(0.94181194,0.19831596){\color[rgb]{0,0,0}\makebox(0,0)[lt]{\lineheight{1.25}\smash{\begin{tabular}[t]{l}$\ppi$\end{tabular}}}}%
    \put(0.442191,0.00462215){\color[rgb]{0,0,0}\makebox(0,0)[lt]{\lineheight{1.25}\smash{\begin{tabular}[t]{l}$X\times\{\pi/2\}$\end{tabular}}}}%
    \put(0.11035449,0.06414078){\color[rgb]{0,0,0}\makebox(0,0)[lt]{\lineheight{1.25}\smash{\begin{tabular}[t]{l}$X\times\{t_1\}$\end{tabular}}}}%
    \put(0.27008218,0.0162113){\color[rgb]{0,0,0}\makebox(0,0)[lt]{\lineheight{1.25}\smash{\begin{tabular}[t]{l}$X\times\{t_2\}$\end{tabular}}}}%
    \put(0.16598233,0.34308028){\color[rgb]{0,0,0}\makebox(0,0)[lt]{\lineheight{1.25}\smash{\begin{tabular}[t]{l}$X\times\{\widetilde{t}_1\}$\end{tabular}}}}%
    \put(0,0){\includegraphics[width=\unitlength,page=2]{rigidity_of_suspensions.pdf}}%
  \end{picture}%
\endgroup%

    \caption{In this diagram, $\mu$ is supported on the black dots, $\widetilde{\mu}$ is supported on the grey ones, and $\nu$ is supported on the endpoints of the suspension, $\mathbf{0}$ and $\ppi$, in such a way that $\mu\in\Int(\widetilde{\mu},\nu)$. Any other measure $\mu'\in\Int(\widetilde{\mu},\nu)$ is atomic, and only differs from $\mu$ by the mass it gives to the points $\{[x_{i1},t_1]\}_{i=1}^{k_1}$, $\{[x_{i1},t_2]\}_{i=1}^{k_1}$, $\{[x_{i2},t_1]\}_{i=1}^{k_2}$, and $\{[x_{i2},t_2]\}_{i=1}^{k_2}$, i.e.\ $\mu'$ may give positive mass to the white dots.}
    \label{fig:rigidity of suspensions}
\end{figure}

Observe now that, if $\gamma$ is a geodesic in $\Susp(X)$ with $\gamma_0={\mathbf{0}}$, $\gamma_1={\ppi}$, then $\proj_{\gamma\#}\Phi(\mu)=\proj_{\gamma\#}\mu$ by lemma \ref{lemma.projectionintogeodics}. Suppose that $\Phi(\mu)\neq\mu$ and we choose $x_o\in X$ satisfying Condition \ref{cond:general position} with respect to $\{t_1,\dots,t_n\}$ and $\{x_{ij}:1\leq j\leq n,\ 1\leq i\leq k_j\}$. Let us verify that $\proj_{\gamma^o\#}\Phi(\mu)\neq\proj_{\gamma^o\#}\mu$ for the geodesic $\gamma^o$ given by $\gamma^o_t= [x_o,t]$. Note first that both $\proj_{\gamma^o\#}\mu$ and $\proj_{\gamma^o\#}\Phi(\mu)$ are atomic measures supported on
\[
\{[x_o,\tan^{-1}(\tan(t_j)\cos(d(x_o,x_{ij})))]:1\leq j\leq n,\ 1\leq i\leq n,\  1\leq i\leq k_{j}\},
\]
giving the same mass to points $[x_o,\tan^{-1}(\tan(t_j)\cos(d(x_o,x_{ij'})))]$ with $3\leq j\leq n$ and $1\leq i\leq k_{j}$. 
On the other hand,
\[
\proj_{\gamma^o\#}\mu([x_o,\tan^{-1}(\tan(t_j)\cos(d(x_o,x_{ij})))]) = \mu([x_{ij},t_j])
\]
and 
\[
\proj_{\gamma^o\#}\Phi(\mu)([x_o,\tan^{-1}(\tan(t_j)\cos(d(x_o,x_{ij})))]) = \Phi(\mu)([x_{ij},t_j])
\]
are different for some $j\in \{1,2\}$ and $1\leq i\leq k_j$, by equation \eqref{eq:intermediate-atomic-measure}. This is a contradiction, and therefore $\Phi(\mu)=\mu$.

For the general case, where $t_j=\pi/2$ is allowed, simply observe that any $\mu$ of the form \eqref{eq:atomic measure} can be approximated with measures of the form \eqref{eq:atomic measure} with $t_j\neq \pi/2$ for all $1\leq j\leq n$. The conclusion then follows from the continuity of $\Phi$.
\end{proof}

Theorem~\ref{t:rigidity of compact spaces} now follows.

\begin{proof}[Proof of theorem~\ref{t:rigidity of compact spaces}]
 The result follows from lemmas \ref{prop.invarianceDiracdeltas} and \ref{prop:invariance of atomic measures in general position}, and the fact that totally atomic measures of the form \eqref{eq:atomic measure} are dense in $\PP_p(\Susp(X))$.
\end{proof}

\medskip
\printbibliography 

\end{document}